\documentclass[11pt]{amsart}
\usepackage{amsmath}
\usepackage{amssymb}
\usepackage{amsfonts}
\usepackage{amsthm}
\usepackage{enumerate}
\numberwithin{equation}{section}

\newtheorem{theorem}{Theorem}[section]
\newtheorem{lemma}[theorem]{Lemma}
\newtheorem{proposition}[theorem]{Proposition}
\newtheorem{corollary}[theorem]{Corollary}

\theoremstyle{definition}

\newtheorem{example}[theorem]{Example}

\theoremstyle{remark}
\newtheorem{remark}[theorem]{Remark}

\numberwithin{equation}{section}

\newcommand{\Wip}{\mathrm{A}_+^1}

\newcommand{\vanish}[1]{\relax}

\newcommand{\N}{\mathbb{N}}

\newcommand{\R}{\mathbb{R}}
\newcommand{\C}{\mathbb{C}}

\newcommand{\ud}{\mathrm{d}}

\newcommand{\eM}{\mathrm{M}}

\newcommand{\Lap}{\mathcal{L}}

\newcommand{\Sum}[2][\relax]{%
 \ifx#1\relax \sideset{}{_{#2}}\sum
 \else \sideset{}{^{#1}_{#2}}\sum
 \fi}

\DeclareMathOperator{\re}{Re}

\newcommand{\abs}[1]{\left| #1 \right|}
\renewcommand{\abs}[1]{\left\vert#1\right\vert}

\DeclareMathOperator{\Sect}{Sect}

\DeclareMathOperator{\Reg}{Reg}

\newcommand{\pfeil}{\longrightarrow}

\DeclareMathOperator{\dom}{dom}
\DeclareMathOperator{\ran}{ran}

\newcommand{\cls}[1]{\overline{#1}}

\DeclareMathOperator{\Lin}{\mathcal{L}}

\newcommand{\norm}[2][\relax]{%
   \ifx#1\relax \ensuremath{\left\Vert#2\right\Vert}
   \else \ensuremath{\left\Vert#2\right\Vert_{#1}}
   \fi}

\makeatletter
\newcommand{\sprod}[2]{\ensuremath{%
  \setbox0=\hbox{\ensuremath{#2}}
  \dimen@\ht0
  \advance\dimen@ by \dp0
  \left(\left.#1\rule[-\dp0]{0pt}{\dimen@}\,\right|#2\hspace{1pt}\right)}}
 \makeatother

\newcounter{aufzi}

\newcounter{aufzii}

\newcounter{aufziii}

\numberwithin{equation}{section}

\def\Reg{\operatorname{Reg}}

\begin{document}


\title[Product formulas in functional  calculi]
{Product formulas in functional calculi for sectorial operators}

\author{Charles Batty}
\address{St. John's College\\
University of Oxford\\
Oxford OX1 3JP, UK
}

\email{charles.batty@sjc.ox.ac.uk}

\author{Alexander Gomilko}
\address{Faculty of Mathematics and Computer Science\\
Nicolas Copernicus University\\
ul. Chopina 12/18\\
87-100 Toru\'n, Poland } \email{gomilko@mat.uni.torun.pl}


\author{Yuri Tomilov}
\address{Faculty of Mathematics and Computer Science\\
Nicolas Copernicus University\\
ul. Chopina 12/18\\
87-100 Torun, Poland\\
and
Institute of Mathematics\\
Polish Academy of Sciences\\
\'Sniadeckich 8\\
00-956 Warsaw, Poland
}

\email{tomilov@mat.uni.torun.pl}

\thanks{The research described in this paper was supported by the EPSRC grant EP/J010723/1.  The second and third authors were also partially supported by the NCN grant DEC-2011/03/B/ST1/00407 and by the EU Marie Curie IRSES program, project ``AOS'', No.\ 318910}

\subjclass{Primary 47A60; Secondary 33C99 47D03}

\keywords{Banach space, sectorial operator, functional 
calculus, product formula, generalised Stieltjes functions, Bernstein functions}

\date{10 September 2014}

\begin{abstract}
We study the product formula $(fg)(A) = f(A)g(A)$ in the framework of (unbounded) functional calculus of sectorial operators $A$. We give an abstract result, and, as corollaries, we obtain new product formulas for the holomorphic functional calculus, an extended Stieltjes functional calculus and an extended
Hille-Phillips functional calculus.  Our results generalise previous work of Hirsch, Martinez and Sanz, and Schilling.
\end{abstract}

\maketitle

\section{Introduction}

Sums and products of sectorial operators on a complex Banach space $X$ arise frequently in the theory of linear evolution equations, but they can be awkward to handle.  It is often unclear whether the sum $A+B$ and the product (composition) $AB$ of two unbounded operators $A$ and $B$, with their natural domains, are closed, or even closable.  It is natural to assume that $A$ and $B$ commute (in the sense of having commuting resolvents), and to exclude cases where the sum or product may fail to be closed due to cancellation, but even then there are difficulties.  For sums of two commuting sectorial operators there are now several theorems, known as Dore-Venni theorems (see, for example,  \cite[Theorem 12.13]{Weis}, \cite{LLL}), where sectoriality of the sum is established, and several of them have analogues for the product.  However these results rely on assuming that at least one of the operators has additional properties such as bounded $H^\infty$-calculus on a sector.  

In this paper, we consider this question for products in a slightly different form.  We take one sectorial operator $A$ and two functions $f$ and $g$ such that $f(A)$ and $g(A)$ can be defined by any of several different functional calculi.  We then try to make sense of the product formula
\begin{equation} \label{main}
(f  g)(A) =  f(A)g(A).\end{equation}
If  $(f  g)(A)$ is defined as a closed operator within some functional calculus and (\ref{main}) is true, then we obtain as a corollary that the product operator $f(A)g(A)$ is closed.

There have to be some supplementary restrictions for results of this type.  For any sectorial operator $A$, the product operator
$A(1+A)^{-1}$ is the bounded operator $1 - (1+A)^{-1}$ with domain $X$, while $(1+A)^{-1}A$ is the restriction of that operator to $\dom(A)$.  So \eqref{main} holds for $f(z) = z$ and $g(z) = (1+\nobreak z)^{-1}$, but not for $f(z)=(1+z)^{-1}$ and $g(z)=z$. Thus the supplementary assumptions may be asymmetrical between $f$ and $g$.

There are some results in the literature where \eqref{main} is established under various assumptions.  Hirsch introduced functional calculus for the convex cone $\mathcal{CBF}$ of complete Bernstein functions, and he proved that \eqref{main} holds if $f$, $g$, and $f g$ all belong to $\mathcal{CBF}$ \cite[Theorem 1]{HirFA}.   Martinez and Sanz \cite{Mart} (see also \cite[Chapter 4]{Mart1}) extended the functional calculus to the linear span $\mathcal T$ of $\mathcal{CBF}$.   Such a function  $f$ has a representation
\begin{equation}
f(z)=a+\int_0^\infty \frac{z \mu(ds)}{1+zs}, \qquad z\in \C\setminus (-\infty,0],
\label{sfun0}
\end{equation}
for some $a\in \C$ and a suitable complex measure $\mu$ on $[0,\infty)$. In particular, if $a\ge 0$ and $\mu$
is a positive measure, then $f$ is a complete Bernstein function
\cite[Chapter 6]{SSV}.  When $A$ is sectorial $f(A)$ is defined as the closure of the operator
\begin{equation}
 \dom(A)\ni x \,\mapsto\, \hat{f}(A)x=
ax+\int_{0}^\infty A(1+sA)^{-1}x\,\mu(ds).
\label{stilclos}
\end{equation}
Martinez and Sanz  proved the following extension of Hirsch's result.

\begin{theorem}\cite[Theorem 2.2]{Mart}.\label{prodM}
If $f$, $g$ and $f g$ are functions of the class $\mathcal{T}$ such
that $f$ has no zeros in $\C \setminus (-\infty,0]$ and
\begin{equation*}
\tilde{f}(z):=1/f(1/z)\in \mathcal{T},
\end{equation*}
then \eqref{main} holds.  In particular, \eqref{main} holds if $f \in \mathcal{CBF}$, $g \in \mathcal{T}$ and $fg \in \mathcal{T}$.
\end{theorem}

Schilling proved another result of similar type in \cite{Sch} (see \cite[Theorem 12.22]{SSV}).  If $-A$ generates a bounded $C_0$-semigroup, then $f(A)$ can be defined for every Bernstein function $f$, and then \eqref{main} holds if $f$, $g$ and $f g$ are all Bernstein functions. 
  
In the results above, none of the classes of functions is  closed under products, so the function $f  g$ has to be assumed to be in the same class in each case.  In this article, we generalise Theorem \ref{prodM} to cases when $f g\not\in \mathcal{T}$, with a much simpler proof than those given in \cite{HirFA} and \cite{Mart}.  To  extend the product formula (\ref{main}) to this case we need a proper definition of the operator $(f g)(A)$, and we use an extension procedure for an elementary functional calculus as described in 
\cite{Ha06}.  This is presented in Section \ref{ABFprel} in a very abstract context, and in later sections for several different functional calculi for sectorial operators.  In doing so, we extend these functional calculi to algebras (so that we no longer have to assume that $f  g$ is in the appropriate class of functions), and we also need to establish that the various calculi are consistent with each other.   Such properties are also needed for applications to rates of decay of operator semigroups as studied in \cite{BCT}.

The product formula is shown to be true, under mild restrictions, in the context of abstract functional calculus (Theorem \ref{MjA}), holomorphic functional calculus (Theorem \ref{Mj}, Corollary \ref{cor1}), and extended Stieltjes calculus (Theorem \ref{cor3}) for sectorial operators, and also the extended Hille-Phillips calculus (Theorem \ref{generator}) for negative generators of bounded $C_0$-semigroups.

\subsection*{Preliminaries}\label{notations}
Throughout, $X$ will be a complex Banach space, $\Lin(X)$ will denote the space of bounded linear operators on $X$, and the identity operator on $X$ will be denoted by $1$.

For a linear operator $A$ on $X$ we denote by $\dom(A)$, $\ran(A)$, $\ker(A)$, and $\sigma(A)$ the
{\em domain}, the {\em range}, the {\em kernel}, and the {\em
spectrum} of $A$, respectively. If $A$ is closable, we shall denote the closure of $A$ by $\overline{A}$. 
If $A$ is injective, we consider the operator $A^{-1}$ with $\dom(A^{-1}) = \ran(A)$.  For operators $A$ and $B$ on $X$, we take the sum $A+B$ and product (composition) $AB$ to have domains
\begin{eqnarray*}
\dom(A+B) &=& \dom(A) \cap \dom(B), \\
\dom(AB) &=& \{x \in \dom(B): Bx \in\dom(A)\}.
\end{eqnarray*}

An operator $A$ is  {\em sectorial} or {\em non-negative} if $A$ is densely defined, $\sigma(A) \subset \C \setminus (-\infty,0)$ and 
\begin{equation}
M(A):=\sup_{s>0}\,\|(1+sA)^{-1}\|<\infty.
\label{stiloper}
\end{equation}
We note that in some places in the literature sectorial operators are not required to be densely defined, or are required to be injective.   A few of our results include an assumption that $A$ is injective.

The following properties of a sectorial operator $A$ are easily seen and well known \cite{Ha06}, \cite{Weis}, \cite{Mart1}:
\begin{enumerate}[\rm (a)]
\item  There exists $\omega \in (0,\pi)$ such that $\sigma(A)\subset \overline{S}_\omega$ and
\[
M(A,\omega):=
\sup\{\|z (z-A)^{-1}\|: z\not\in \overline{S}_{\omega}\}<\infty,
\]
where
\[
S_\omega:=\{z\in \C:\,z\not=0,\;|\arg z|<\omega\}
\]
denotes the open sector which is symmetric about the positive real axis with half-angle $\omega$.  Then we write $A \in \operatorname{Sect}(\omega)$.
\item  For any $\delta>0$, $A + \delta$ is sectorial and invertible.  Moreover, for any $s>0$,
\begin{equation}\label{sA02}
\|(1+\delta s+sA)^{-1}\|=(1+\delta s)^{-1}\|(1+(1+\delta s)^{-1} sA)^{-1}\|\le \frac{M(A)}{1+\delta s} \,.
\end{equation}

\item  For any $s>0$,
\begin{equation}
\|sA(1+sA)^{-1}\|=\|1-(1+sA)^{-1}\|\leq 1+M(A).
\label{sA0}
\end{equation}

\item  Let $n\in\N$.  In the strong operator topology,
\begin{equation} \label{sot}
\lim_{s\to0+} (1+sA)^{-n} = 1.
\end{equation}
\end{enumerate}

\section{Extended abstract functional calculus}
\subsection{Abstract functional calculus}\label{ABFprel}
Here we recall the abstract, purely algebraic, notion of functional calculus described in \cite{deLau95}, \cite{Ha05} and \cite[Chapter 1]{Ha06}.  One starts with a triple
$(\mathcal{E},\mathcal{F},\Phi)$, where
$\mathcal{F}$ is a  complex, commutative algebra with unit $\bf{i}$,
$\mathcal{E}$ is a subalgebra of $\mathcal{F}$ and
$\Phi: \mathcal{E}\mapsto \Lin(X)$ is
an algebra homomorphism.
This situation is called an {\em abstract functional calculus} over $X$.

For $f\in \mathcal{F}$ each member of the set
\[
\Reg(f):=\{e\in \mathcal{E}: \text{$e f\in \mathcal{E}$, $\Phi(e)$ injective}\}
\]
is called a {\em regulariser} for $f$.
If $\Reg(\bf{i})\not=\emptyset$, then the abstract functional calculus is called {\em proper}.
For a proper abstract functional calculus
one can extend $\Phi$ to the subalgebra
\[
\mathcal{F}_r:=\{f\in \mathcal{F}:\, \mbox{Reg}(f)\not=\emptyset\},
\]
which necessarily contains $\mathcal{E}$ and ${\bf i}$, by setting
\[
\Phi(f):=\Phi(e)^{-1}\Phi(ef),
\]
where $e\in \Reg(f)$ is arbitrary.
This does not depend on the choice of $e$
and yields a closed (in general unbounded) operator $\Phi(f)$, which coincides with the original $\Phi(f)$
when $f$ itself is in $\mathcal{E}$.
We define
\[
\mathcal{F}_b=\{f\in \mathcal{F}_r:\,\Phi(f)\in \mathcal{L}(X)\},
\]
the set of all those regularisable functions which give rise to bounded operators.  Note that ${\bf i} \in \mathcal{F}_b$ and $\Phi({\bf i}) = 1$.

In many examples of proper abstract functional calculi, including those in later sections of this paper, $\mathcal{F}$ is an algebra of functions on a set $\Omega\subset\C$ and the identity  function $\iota :z\mapsto z$ is regularisable.  Then we can form the operator $A:=\Phi(\iota)$.  In this case  we call the abstract functional calculus a {\em functional calculus for} $A$, and we may write $f(A)$ instead of $\Phi(f)$.

The following proposition summarises some fundamental properties of any proper
abstract functional calculus $(\mathcal{E}, \mathcal{F},\Phi)$ over a Banach space $X$
(see \cite{deLau95}, \cite{Ha05}, \cite[Chapter 1]{Ha06}).

\begin{proposition}\label{properF}
Let $(\mathcal{E}, \mathcal{F},\Phi)$ be a proper abstract functional calculus over a Banach space $X$, and let $f,g \in \mathcal F_r$.  Then
\begin{enumerate}[\rm(a)]
\item \label{sum} $\Phi({\bf i})=1$ and
\[
\Phi(f)+\Phi(g)\subset \Phi(f+g).
\]
\item If $g \in \mathcal F_b$, then
\[
\Phi(f)+\Phi(g)= \Phi(f+g).
\]
\item \label{Baa} $\Phi(f)\Phi(g)\subset \Phi(f g)$ and
\[
\dom(\Phi(f)\Phi(g))=\dom(\Phi(fg))\cap \dom(\Phi(g)).
\]
\item \label{grrr} If $g\in \mathcal{F}_b$, then
\[
\Phi(g)\Phi(f)\subset \Phi(f)\Phi(g)= \Phi(f g).
\]
\item \label{errr} If $e\in \Reg(f)$, then $\ran(\Phi(e))\subset \dom(\Phi(f))$.
\item \label{inverrr} If $fg={\bf i}$, then $\Phi(f)$ is injective and
\[
[\Phi(f)]^{-1}=\Phi(g).
\]
\item 
If $g\in \mathcal{F}_b$ and $\Phi(g)$ is injective, then
\[
\Phi(f)=[\Phi(g)]^{-1}\Phi(f) \Phi(g).
\]
\end{enumerate}
\end{proposition}

\begin{remark}\label{RemPr}
It follows from (\ref{Baa}) that
\begin{equation}\label{RKrit}
\Phi(f)\Phi(g)=\Phi(f g)\;\Leftrightarrow
\;\dom(\Phi(fg))\subset \dom(\Phi(g)).
\end{equation}
\end{remark}

The next proposition is a slight generalisation of
\cite[Proposition 3.3(v)]{Ha05} and \cite[Proposition 3.2]{Clark}.  It establishes a weak form of the product formula under mild assumptions.

\begin{proposition}\label{GenerC}
Let $(\mathcal{E}, \mathcal{F},\Phi)$ be a proper abstract functional calculus over a Banach space $X$.
\begin{enumerate}[\rm (a)]
\item
Let $f \in \mathcal F_r$, and $\tilde X$ be a subspace of $\dom(\Phi(f))$.  Assume that there exists a sequence
$(e_n)_{n\ge 1}\in \mathcal{F}_b$ such that $\Phi(e_n)\to 1$ strongly as $n\to\infty$, and
\[
\ran(\Phi(e_n))\subset \tilde{X},\quad n\in \N.
\]
Then $\tilde{X}$ is a core for $\Phi(f)$.

\item Let $f,g\in \mathcal{F}_r$, and assume that there exist sequences
$(e_n)_{n\ge1}$ in $\Reg(f)$, and $(\tilde{e}_n)_{n\ge1}$ in $\Reg(g)$, such that
\[
e_nf\in \mathcal{F}_b,\quad
\tilde{e}_ng\in \mathcal{F}_b,\quad n\in \N,
\]
and
\[
\Phi(e_n)\Phi(\tilde{e}_n)=\Phi(e_n \tilde{e}_n)\to 1\quad \mbox{strongly as}\quad n\to\infty.
\]
Then
\[
\overline{\Phi(f)\Phi(g)}=\Phi(fg).
\]
\end{enumerate}
\end{proposition}

\begin{proof}
(a)  Let $x\in \dom(\Phi(f))$. Then by Proposition \ref{properF}(\ref{grrr})
we have
\[
\Phi(e_n)\Phi(f)x=\Phi(f)\Phi(e_n)x.
\]
Then
\[
y_n:=\Phi(e_n)x\in \ran(\Phi(e_n))\subset\tilde{X}\]
and
\[
y_n\to x,\quad
\Phi(f)y_n \to \Phi(f)x \quad \mbox{as}\quad n \to\infty.
\]
So, $\tilde{X}$ is a core for $f(A)$.

(b) By Proposition \ref{properF}(\ref{Baa}), $\dom(\Phi(f)\Phi(g)) \subset \dom(\Phi(fg))$.  
By Proposition \ref{properF}(\ref{errr}) we have
\[
\ran(\Phi(e_n\tilde{e}_n))=\ran(\Phi(e_n)\Phi(\tilde{e}_n))\subset \ran(\Phi(e_n))\subset \dom(\Phi(f)),
\]
and by symmetry we also have
\[
\ran(\Phi(e_n\tilde{e}_n))\subset \dom(\Phi(g)).
\]
Moreover, by Proposition \ref{properF}(\ref{grrr}) we have
\[
\Phi(g)\Phi(e_n\tilde{e}_n)=\Phi(g)\Phi(\tilde{e}_n)\Phi(e_n)=\Phi(g\tilde{e}_n)\Phi(e_n)
=\Phi(e_n)\Phi(g\tilde{e}_n),
\]
so
\[
\ran(\Phi(g)\Phi(e_n\tilde{e}_n))\subset \ran(\Phi(e_n))\subset \dom(\Phi(f)),
\]
and hence
\[
\ran(\Phi(e_n\tilde{e}_n))\subset\dom(\Phi(f)\Phi(g))\subset \dom(\Phi(fg)).
\]
Now (a) implies that $\dom(\Phi(f)\Phi(g))$ is a core for $\Phi(fg)$.
\end{proof}

The next proposition is an abstract result concerning two functional calculi which are related by a quotient structure.  It will enable us to apply a quotient construction to reduce some of the proofs of results for sectorial operators $A$ in later sections to the case when $A$ is injective.  Injectivity is often assumed when studying sectorial operators but this constraint is sometimes unnatural.

\begin{proposition}\label{quotfc}
Let $(\mathcal{E},\mathcal{F},\Phi)$ and $(\mathcal{E},\mathcal{F},\Phi_0)$ be proper abstract functional calculi on Banach spaces $X$ and $X_0$, respectively, and let $u: X \to X_0$ be a bounded linear surjection such that
\[
\Phi_0(g) u = u \Phi(g), \qquad g \in \mathcal{E}.
\]
Let $f \in \mathcal{F}$, and assume that $e\in\mathcal{E}$ is a regulariser for $f$ with respect to $\Phi$, and that $\Phi(e)$ maps $\ker u$ onto $\ker u$.  Then $e$ is a regulariser for $f$ with respect to $\Phi_0$, and
\[
\Phi_0(f)u = u \Phi(f), \qquad \dom(\Phi_0(f)) = \{u(x): x \in \dom(\Phi(f))\}.
\]
\end{proposition}

\begin{proof}
Assume that $\Phi_0(e)u(x) = 0$ for some $x \in X$.  Then $u\Phi(e)x=0$, so $\Phi(e)x \in \ker u$.  By assumption,  $\Phi(e)x = \Phi(e)y$ for some $y \in \ker u$.  Since $\Phi(e)$ is injective, $x=y$, and so $u(x)=0$.  Thus $\Phi_0(e)$ is injective, and hence $e$ is a regulariser for $f$ with respect to $\Phi_0$.

Now consider $u(x), u(y)$ for arbitrary $x,y \in X$.   Then
\begin{align*}
\lefteqn{u(x) \in \dom(\Phi_0(f)), \Phi_0(f)u(x)=u(y)} \\
&\iff \Phi_0(ef)u(x) = \Phi_0(e)u(y) \\
&\iff u\Phi(ef)x = u\Phi(e)y \\
&\iff \text{$\Phi(ef)x = \Phi(e)y + z$ \quad for some $z \in \ker u$} \\
&\iff \text{$\Phi(ef)x = \Phi(e)(y+z')$ \quad for some $z' \in \ker u$} \\
&\iff \text{$x\in \dom(\Phi(f)), \Phi(f)x = y+z'$ \quad for some $z' \in \ker u$} \\
&\iff x \in \dom(\Phi(f)), u\Phi(f)x = u(y).
\qedhere
\end{align*}
\end{proof}

\subsection{Product formula in abstract functional calculus}
Let $(\mathcal{E},\mathcal{F},\Phi)$ be a proper abstract functional calculus,
 with subalgebras $\mathcal F_b \subset \mathcal{F}_r\subset \mathcal{F}$ as in Section \ref{ABFprel}.

\begin{theorem}\label{MjA} Let  $f,g\in \mathcal{F}_r$, with $f$ having an inverse $f^{-1}$ in $\mathcal F_r$.  Assume that there exist elements $e_1,e_2,e_3 \in \mathcal{F}_b$, with inverses in $\mathcal F_r$, such that
\[
e_1  f\in \mathcal{F}_b,\quad e_2 g\in \mathcal{F}_b, \quad e_3 f^{-1} \in \mathcal F_b.
\]
Assume also that there is a complex polynomial $p(z)$ with $p(0)\ne0$ such that
\begin{equation}
\ran (\Phi(p(e_3))) \subset \ran(\Phi(e_1 e_2)).
\label{nnA}
\end{equation}
Then the product formula
\begin{equation}\label{mainA}
\Phi(f)\Phi(g)=\Phi(f g)
\end{equation}
 holds.
\end{theorem}

\begin{proof}
By Remark \ref{RemPr} it is enough to prove that
\begin{equation}\label{inclusA}
\dom(\Phi(f g))\subset \dom(\Phi(g)).
\end{equation}
We define the bounded operators
(see Proposition \ref{properF}(\ref{grrr}))
\[
F:=\Phi(f)\Phi(e_1)=\Phi(f e_1),\quad G:=\Phi(g)\Phi(e_2)=\Phi(ge_2),
\]
and
\[
\quad Q:=\Phi(f^{-1})\Phi(e_3)=\Phi(f^{-1}e_3).
\]
Let $x \in \dom(\Phi(fg))$.  By Proposition \ref{properF}(\ref{grrr}) we have
\[
\Phi(e_1 e_2) \Phi(f g)x=\Phi(f g) \Phi(e_1 e_2)x,
\]
and
\begin{eqnarray*}
F G x=\Phi(f e_1)\Phi(g e_2)x &=& \Phi((f g) ( e_1 e_2))x \\
&=& \Phi(fg)\Phi(e_1e_2)x \\
&=& \Phi(e_1 e_2) \Phi(f g)x\in \ran(\Phi(e_1 e_2)).
\end{eqnarray*}
Since
\[
Q\Phi(e_1  e_2)=\Phi(e_1 e_2)Q
\]
by Proposition \ref{properF}(\ref{grrr}), it follows that
\begin{equation}\label{e1e2A}
\Phi(e_3) \Phi(e_1) G x=\Phi((f^{-1} e_3) (f e_1)  (ge_2))x
=Q F G x \in \ran(\Phi(e_1 e_2)).
\end{equation}
Since $p(0)\ne0$, there exist a polynomial $\tilde p(z)$ and $\alpha \in \C$ such that
$$
\alpha p(z)  + z \tilde p(z) = 1.
$$
Putting $y = \Phi(e_1)Gx$, we have
$$
y = \alpha p(\Phi(e_3)) y + \tilde p(\Phi(e_3)) \Phi(e_3)y \in \ran(\Phi(e_1 e_2)),
$$
using the assumption (\ref{nnA}) and (\ref{e1e2A}). 
Moreover, the operators $\Phi(e_1)$, $\Phi(e_2)$ are injective by
Proposition \ref{properF}(\ref{inverrr}). Then we obtain
that
\[
\Phi(e_1) \Phi(g) \Phi(e_2)x=\Phi(e_1) G x\in  \ran(\Phi(e_1  e_2))
=\ran(\Phi(e_1)\Phi_2(e_2)),
\]
and  this gives
\[
\Phi(g) \Phi(e_2)x\in \ran(\Phi(e_2)),\quad  x\in \dom(\Phi(fg)).
\]
Since $\Phi(g e_2)x = \Phi(g)\Phi(e_2)x$ by Proposition  \ref{properF}(\ref{grrr}), we obtain that
$x\in \dom(\Phi(g))$ and then we have the inclusion
(\ref{inclusA}).
\end{proof}

The assumed invertibility of $e_j$ was used in the proof of Theorem \ref{MjA} only to ensure that $\Phi(e_j)$ is injective.  Hence the following variant is also true.

\begin{corollary}\label{MjA00} Let  $f$, $g\in \mathcal{F}_r$
and $f^{-1}\in \mathcal{F}_r$ with regularisers
$e_1, e_2, e_3\in \mathcal{E}$, respectively, so
\[
e_1  f\in \mathcal{E},\quad e_2 g\in \mathcal{E},\quad
e_3 f^{-1}\in \mathcal{E},
\]
and the bounded operators $\Phi(e_j)$ are injective for $j=1,2,3$. Assume that
the condition \eqref{nnA} holds. Then the product formula \eqref{mainA} is true.
\end{corollary}

\section{Holomorphic calculus} \label{secthol}

We will apply the results of the previous section to the
holomorphic calculus for sectorial operators. This functional calculus is described in detail in \cite{Ha06}, \cite{Weis} and other texts, and we give only a short summary here.

For $\phi \in (0,\pi)$, let
\[
H_0^\infty(S_\phi):=
\left\{f\in \mathcal{O}(S_\phi) : \text{$|f(z)|\leq C\min(|z|^s, |z|^{-s})$ for some $C,s>0$}\right\},
\]
where $\mathcal{O}(S_\phi)$ denotes
the space of all holomorphic functions
on the sector $S_\phi$.
Let
\begin{equation*}
\tau(z):=\frac{z}{(1+z)^2}
\end{equation*}
and
\begin{eqnarray*}
\mathcal{B}(S_\phi)& :=& \left\{f :  \text{$\tau^n f\in H_0^\infty(S_\phi)$ for some $n\in\N$}\right\} \\
&=& \left\{f\in \mathcal{O}(S_\phi) : \text{$|f(z)|\leq C\max \left(|z|^s, |z|^{-s} \right)$ for some $C,s>0$}\right\}.
\end{eqnarray*}

Let  $A\in
\operatorname{Sect}(\omega)$ for some angle $\omega$, and let $\omega<\phi<\pi$.   Let
\[
\mathcal{E}=H_0^\infty(S_\phi),\quad
\mathcal{F}=\mathcal{O}(S_\phi).
\]
For $f \in H_0^\infty(S_\phi)$, we define
\begin{equation}\label{Cauchy}
\Phi(f) = f(A):=\frac{1}{2\pi i}\int_\Gamma f(z) (z-A)^{-1}\,dz,
\end{equation}
where $\Gamma$ is the downward oriented boundary of a sector
$S_{\omega_0}$, with $\omega<\omega_0<\phi$.  This definition is independent of $\omega_0$, and
\[
\Phi: H_0^\infty(S_\phi)\mapsto \mathcal{L}(X),\quad \Phi(f)=f(A),
\]
is an algebra homomorphism, with $\Phi(\tau)=A(1+A)^{-2}$. 

Now assume that $A$ is injective, so $\Phi(\tau)$ is injective.  Then we have a proper abstract functional calculus for $A$, and the extended calculus is called the {\em holomorphic calculus} for $A$.  Any function  $f \in \mathcal B(S_\phi)$ has a regulariser of the form $\tau^n$, and
\begin{equation*}
f(A)=[\tau^n(A)]^{-1}(\tau^n f)(A),
\end{equation*}
where $n\in\N$ is large enough to ensure that
\[
\tau^n f\in H_0^\infty (S_\phi).
\]

This functional calculus formally depends on a choice of $\phi$, but the calculi are consistent when we identify a function $f$ on $S_\phi$ with its restriction to $S_\psi$ for $\omega < \psi < \phi$.  We may therefore make this identification and consider our holomorphic calculus to be defined on the algebra
\[
\mathcal{B}[S_\omega]:=\bigcup_{\omega<\phi<\pi}\,\mathcal{B}(S_\phi).
\]

We define
\[
H(A):=\left\{f\in \mathcal{B}[S_\omega] : f(A) \in \mathcal{L}(X)\right\}.
\]
The elementary product rule (Proposition \ref{properF}(\ref{grrr})) says that
\begin{equation}\label{zv}
g(A)f(A)\subset f(A)g(A)=[f g](A),\qquad f \in \mathcal{B}[S_\omega],\,g\in H(A).
\end{equation}
From Theorem \ref{MjA} we obtain the following product rule for the extended
holomorphic calculus.

\begin{theorem}\label{Mj}
Let $A\in \operatorname{Sect}(\omega)$ be injective,
$\omega<\phi<\pi$ and $f$, $g\in \mathcal{B}(S_\phi)$.
Assume that $f$ has no zeros in $S_\phi$ and $1/f\in \mathcal{B}(S_\phi)$, and that there exist $e_1,e_2,e_3\in H(A)$, and a complex polynomial $p(z)$ with $p(0)\ne0$, such that 
\begin{enumerate}[\rm(i)]
\item for $j=1,2,3$, $e_j$ has no zeros in $S_\phi$ and $1/e_j\in \mathcal{B}(S_\phi)$,
\item
\[
e_1 f\in H(A),\quad e_2 g\in H(A), \quad
\frac{e_3}{f}\in H(A),
\]
\item 
\begin{equation}
\ran (p(e_3(A)))\subset \ran((e_1e_2)(A)).
\label{nn}
\end{equation}
\end{enumerate}
Then the product formula \eqref{main} holds.
\end{theorem}

\begin{corollary}\label{cor1}
Let  $A\in \operatorname{Sect}(\omega)$ be injective, $\phi\in
(\omega,\pi)$ and $f$, $g\in\mathcal{B}(S_\phi)$.  Assume that $f$ has no zeros in $S_\phi$ and $1/f\in \mathcal{B}(S_\phi)$, and
\begin{equation}
\frac{f(z)}{(1+z)^k}\in H(A),\quad
\frac{g(z)}{(1+z)^m}\in H(A),\quad
\frac{z^r}{f(z)(1+z)^r}\in H(A),
\label{3cond}
\end{equation}
for some $k,m,r \in \N\cup\{0\}$.
Then the product formula \eqref{main} holds.
\end{corollary}

\begin{proof}
The statement follows from Theorem \ref{Mj}
if we take the functions
\[
\begin{aligned}
e_1(z)&=\frac{1}{(z+1)^k},\quad &e_2(z) &=\frac{1}{(z+1)^m}, \\
e_3(z)&=\frac{z^r}{(z+1)^r}, \quad &p(z) &= (1-z)^{k+m}.
\end{aligned}
\]
The condition (\ref{nn}) holds because
\[
1-A^r(1+A)^{-r} = B(1+A)^{-1} = (1+A)^{-1}B,
\]
where $B$ is a bounded operator, and hence 
\[
\ran\left((1-A^r(1+A)^{-r})^{k+m}\right) \subset \ran \left((1+A)^{-(k+m)}\right).
\qedhere
\]

\end{proof}

Note that we cannot remove the third condition in (\ref{3cond})
of  Corollary \ref{cor1}. For example let  $f,g\in \mathcal{B}(S_\pi)$ be
\[
f(z)=\frac{1}{1+z},\qquad g(z)=z.
\]
Then for any sectorial  operator $A$ we have
\[
f(z)\in H(A),\quad \frac{g(z)}{1+z}\in H(A), 
\]
so the first two conditions in (\ref{3cond}) are satisfied with
$k=0$ and $m=1$.  If $A$ is unbounded, then (as observed in the Introduction)
\begin{equation} \label{fggf}
f(A)g(A)=(1+A)^{-1}A\not=A(1+A)^{-1}=g(A)f(A)=[f g](A).
\end{equation}
This is consistent with Theorem \ref{Mj} because,
for any $r\in \N$,
\[
\frac{z^r}{f(z)(1+z)^r}=\frac{z^r}{(1+z)^{r-1}}\not \in H(A).
\]
On the other hand,
\[
\frac{z}{g(z)(1+z)}=\frac{1}{1+z}\in H(A),
\]
so the final equality in (\ref{fggf}) follows from Corollary \ref{cor1} with $f$ and $g$ interchanged.

\section{Extended Stieltjes calculus}
In this section $A$ will be a sectorial operator
on a Banach space $X$, and we shall construct the (extended) Stieltjes calculus for $A$.  When $A$ is injective, we show that the extended holomorphic calculus and the extended Stieltjes calculus are consistent (Theorem \ref{ConHPH}). 
We also show that an extension of Hirsch's definition (\ref{stilclos}) is a particular case of the extended Stieltjes calculus (Corollary \ref{coins}).  Finally we consider the product formula for the extended Stieltjes calculus (Theorem \ref{cor3}).

\subsection{The Stieltjes algebra}\label{SubGSt}
In order to construct an extended Stieltjes calculus we need first to introduce a suitable algebra of functions.  We begin by reviewing a few properties of generalised Stieltjes functions here; for a detailed account, see \cite{Karp}.

A function
$f:(0,\infty)\mapsto [0,\infty)$ is called a {\it generalised Stieltjes function of order} $\alpha>0$ if it can be represented as
\begin{equation}\label{salpha}
f(z)=a+\int_0^\infty \frac{\mu(dt)}{(z+t)^\alpha},\qquad z>0,
\end{equation}
where $a\ge 0$ and $\mu$ is a (unique) positive Radon measure on $[0,\infty)$
satisfying
\begin{equation}\label{mmu}
\int_0^\infty\frac{\mu(dt)}{(1+t)^\alpha}<\infty.
\end{equation}
The class of generalised Stieltjes functions of order $\alpha$ will be denoted by $\mathcal{S}_\alpha$.

When $\alpha=1$, the class $\mathcal{S}_1$ of generalised Stieltjes functions of order $1$ is simply the class of Stieltjes functions,
and we will write $\mathcal{S}$ in place of $\mathcal{S}_1$, thus following established notation.  A very informative
discussion of Stieltjes functions is contained in \cite[Chapter 2]{SSV}.

There are other representations for Stieltjes
functions in the literature. For example, we note
the representation
\begin{equation}\label{compbern1}
f(z) = \frac{b}{z^\alpha} + \int_0^\infty \frac{\nu(\ud{s})}{(1 +zs)^\alpha},
\qquad \int_0^{\infty}\frac{\nu(\ud{s})}{(1 +s)^\alpha} <\infty.
\end{equation}
The representations (\ref{salpha}) and (\ref{compbern1}) are
equivalent in the sense that one of them is transformed into the other by the
change of variable $s=1/t$ on $(0,\infty)$, with $b=\mu(\{0\})$ and $a = \nu(\{0\})$, and the measures $\mu$ and $\nu$ satisfy the same integrability condition (\ref{mmu})   (see \cite{Karp}).

For $0 < \beta<\alpha$, one has
\begin{equation} \label{fracpow}
z^{-\beta} = \frac{\Gamma(\alpha)}{\Gamma(\beta)\Gamma(\alpha-\beta)} \int_0^\infty \frac{s^{\beta-1}}{(1+zs)^\alpha} \, ds  \;\in \mathcal{S}_\alpha.
\end{equation}
Here and subsequently, $\Gamma$ denotes the gamma function.

If $f$ is a generalised Stieltjes function (of any positive order), then $f$ is decreasing on $(0,\infty)$ and $f$ admits a unique analytic extension to
$\C\setminus (-\infty,0]$, which is given by \eqref{salpha} and will be denoted by the same symbol. 

Each $\mathcal{S}_\alpha$ is a convex cone, and the following inclusions 
hold \cite[Theorem 3]{Karp}:
\begin{equation}\label{Sinclus}
\mathcal{S}_\alpha\,\subset\,\mathcal{S}_\beta,\;\;0<\alpha<\beta.
\end{equation}
For products, a classical result  of Widder \cite[Chapter 7, $\S\,7.4$]{widder} states
that for functions  $f_1\in \mathcal{S}_{\alpha_1}$, $f_2\in \mathcal{S}_{\alpha_2}$ of the particular form
\begin{equation*}
f_j(z)=\int_{0+}^\infty \frac{\mu_j(ds)}{(z+s)^{\alpha_j}},\quad j=1,2,
\end{equation*}
there exists a positive Radon measure $\nu=\nu(\mu_1,\mu_2)$ on $(0,\infty)$ such  that
\[
\int_{0+}^\infty \frac{\nu(ds)}{(1+s)^{\alpha_1+\alpha_2}}<\infty
\]
and
\begin{equation*}
f_1(z) f_2(z)=\int_{0+}^\infty \frac{\nu(ds)}{(z+s)^{\alpha_1+\alpha_2}}, \quad z>0.
\end{equation*}
Thus $f_1 f_2\in \mathcal{S}_{\alpha_1+\alpha_2}$.  In order to include as many functions as possible we need to extend this result to the whole class of generalised Stieltjes functions allowed in (\ref{salpha}).

\begin{lemma}\label{WiddR}
Let $f_j\in \mathcal{S}_{\alpha_j}$ for some $\alpha_j>0$, for  $j=1,2$. Then
$f_1 f_2\in \mathcal{S}_{\alpha_1+\alpha_2}$.
\end{lemma}

\begin{proof}
Let
\begin{equation*}
f_j(z)=a_j+\int_0^\infty \frac{\,\mu_j(ds)}{(z+s)^{\alpha_j}}
=a_j+ \frac{b_j}{z^{\alpha_j}} +\int_{0+}^\infty \frac{\,\mu_j(ds)}{(z+s)^{\alpha_j}},
\end{equation*}
where $b_j=\mu_j(\{0\})\ge 0,\; j=1,2$.  Using Widder's result and
(\ref{Sinclus}), we obtain that
\[
f_1(z)f_2(z)- \frac{b_1}{z^{\alpha_1}}\int_{0+}^\infty  \frac{\mu_2(ds)}{(z+s)^{\alpha_2}}
-\frac{b_2}{z^{\alpha_2}}\int_{0+}^\infty  \frac{\mu_1(ds)}{(z+s)^{\alpha_1}}\in \mathcal{S}_{\alpha_1+\alpha_2}.
\]
Thus it suffices to show that a function
\[
z^{-\beta}\int_{0+}^\infty  \frac{\mu(ds)}{(z+s)^{\alpha}} \in \mathcal{S}_{\alpha+\beta},
\]
when $\alpha,\beta>0$ and $\mu$ satisfies (\ref{mmu}).

From the representation \cite[p.302]{Prudnikov1}
\[
\frac{1}{z^\beta (s+z)^\alpha}
=\frac{c_{\alpha,\beta}}{s^{\alpha+\beta-1}}
\int_0^s \frac{t^{\alpha-1}(s-t)^{\beta-1}\,dt }{(t+z)^{\alpha+\beta}},
\quad
\]
with the constant
\[
c_{\alpha,\beta}=\frac{\Gamma(\alpha+\beta)}{\Gamma(\alpha)\Gamma(\beta)},
\]
it follows  that
\[
z^{-\beta}\int_{0+}^\infty  \frac{\mu(ds)}{(z+s)^{\alpha}}
=c_{\alpha,\beta}\int_{0}^\infty
\left(\int_t^\infty
\frac{t^{\alpha-1}(s-t)^{\beta-1}}{s^{\alpha+\beta-1}}\mu(ds)
\right)\frac{dt}{(z+t)^{\alpha+\beta}}.
\]
This has the form of a function in $\mathcal{S}_{\alpha+\beta}$. 
\end{proof}

If $f \in \mathcal{S}_\alpha$ has the representation (\ref{compbern1}), then 
\[
0 \le z^\alpha f(z) \le b + \int_0^1 \nu(ds) + \int_{1+}^\infty \frac{\nu(ds)}{(1+s)^\alpha}, \qquad 0<z<1.
\]
We will need to restrict to functions $f$ which are bounded at $0$.  
Thus we consider the following subclass of generalised Stieltjes functions
\[
\mathcal{S}_{\alpha,b}=\left\{f\in \mathcal{S}_\alpha: f(z)=\int_0^\infty
\frac{\nu(ds)}{(1+zs)^\alpha},\;\;\nu([0,\infty))<\infty\right\}.
\]
By the Lebesgue dominated convergence theorem, if  $f\in
\mathcal{S}_\alpha$ then $f \in \mathcal{S}_{\alpha,b}$ if and
only if $f$ is bounded on $(0,\infty)$, and
then
\[
\sup_{z>0}\,f(z)=\lim_{z\to 0+}\,f(z)= \nu([0,\infty)).
\]
The next proposition follows from this fact, (\ref{Sinclus}) and Lemma \ref{WiddR}.

\begin{proposition}\label{Sbound}
The inclusions
\[
\mathcal{S}_{\alpha,b}\subset \mathcal{S}_{\beta,b},\quad \alpha<\beta,
\]
and
\[
\mathcal{S}_{\alpha,b}\cdot \mathcal{S}_{\beta,b}
\subset \mathcal{S}_{\alpha+\beta,b},\quad \alpha,\beta>0.
\]
hold.
\end{proposition}

Now we can obtain algebras of functions by taking linear spans and unions of the classes $\mathcal{S}_\alpha$ or $\mathcal{S}_{\alpha,b}$ for $\alpha>0$.  By \eqref{Sinclus} or Proposition \ref{Sbound}, it suffices to take the unions over $\alpha\in\N$, and we shall restrict attention to that case.

For each $n\in \N$, we let $\tilde{\mathcal{S}}_{n}$ and  $\tilde{\mathcal{S}}_{n,b}$ be the complex linear span of $\mathcal{S}_n$ and $\mathcal{S}_{n,b}$, respectively.   Thus $f\in \tilde{\mathcal{S}}_{n,b}$ if and only if $f$ has the form
\begin{equation}\label{compbern1C}
f(z) = \int_0^\infty \frac{\nu(d{s})}{(1 +zs)^n},\quad z>0,
\end{equation}
where $\nu$ is a complex Radon measure of finite total variation $|\nu|$ on $[0,\infty)$.
By Proposition \ref{Sbound},
\[
\tilde{\mathcal{S}}_{n,b}\subset \tilde{\mathcal{S}}_{m,b}, \quad m>n,
\]
and
\[
\tilde{\mathcal{S}}_{n,b}\cdot \tilde{\mathcal{S}}_{m,b} \subset
\tilde{\mathcal{S}}_{n+m,b}, \quad n,m\in \N.
\]
Let
\[
\tilde{\mathcal{S}}_b:=\bigcup_{n\in \N}\,\tilde{\mathcal{S}}_{n,b}.
\]
Then $\tilde{\mathcal{S}}_b$ is an algebra, which we may call the {\it bounded Stieltjes algebra}.

\subsection{Stieltjes functional calculus}
Let $A$ be a sectorial operator on a Banach space
$X$. We define the map $S:\tilde{\mathcal{S}}_b \to \mathcal{L}(X)$
by
\begin{equation}\label{Asfa}
S(f)x:=f(A)x=\int_0^\infty (1+sA)^{-n}x\,\nu(ds),\quad x\in X,
\end{equation}
where
\begin{equation}\label{GlF}
f(z)=\int_0^\infty \frac{\nu(ds)}{(1+zs)^n}\in \tilde{\mathcal{S}}_{n,b}\subset \tilde{\mathcal{S}}_b.
\end{equation}
By (\ref{stiloper}) the integral in \eqref{Asfa} converges in operator-norm and
\begin{equation*}
\|f(A)\|\le M(A)^n\int_{0}^\infty |\nu|(ds),\quad f\in \tilde{\mathcal{S}}_{n,b}.
\end{equation*}

For a given $f$ and $n$, the measure $\nu$ in (\ref{Asfa}) is unique (see \cite{Karp}), and the following proposition shows that the definition of $f(A)$ is independent of $n$.

\begin{proposition}\label{ProPP}
Let $f\in \tilde{\mathcal{S}}_{n,b}$ have the representation \eqref{GlF} for a measure $\nu_n$.  Let $m>n$ and let $\nu_m$ be the corresponding respresenting measure of $f$ as a member of $\tilde{\mathcal{S}}_{m,b}$. Let the bounded
operator $f_n(A)$ be defined by \eqref{Asfa} and let
\[
f_m(A)x:=\int_0^\infty (1+sA)^{-m}x\,\nu_m(ds),\quad x\in X.
\]
Then $f_n(A)=f_m(A)$.
\end{proposition}

\begin{proof}
It is enough to consider the case $m=n+1$.
We consider the measures $\nu_n$ and $\nu_{n+1}$ on $[0,\infty)$ in a similar way to the proof of \cite[Theorem 3]{Karp}.  Note first that $\nu_n(\{0\}) = \lim_{z\to\infty} f(z) = \nu_{n+1}(\{0\}) =: a$.   Replacing $f(z)$ by $f(z)-a$, we may assume that $a=0$.

For $s>0$,
\begin{equation*}
n\int_0^s \frac{t^{n-1}\,dt}{(1+zt)^{n+1}}=
n\int_{1/s}^\infty  \frac{d\tau}{(\tau+z)^{n+1}}
=\frac{s^n}{(1+zs)^n}.
\end{equation*}
Hence, 
\begin{eqnarray*}
f(z) = \int_{0+}^\infty \frac{\nu_n(ds)}{(1+sz)^n}
&=& \int_{0+}^\infty n \int_0^s \frac{t^{n-1}\,dt}{(1+zt)^{n+1}}\frac{\nu_n(ds)}{s^n} \\
&=& \int_{0}^\infty nt^{n-1}
\int_t^\infty \frac{\nu_n(ds)}{s^n}  \frac{dt}{(1+zt)^{n+1}}.
\end{eqnarray*}
The application of Fubini's theorem above is justified by a similar calculation with $\nu_n$ replaced by $|\nu_n|$.  On letting $z\to0+$, this also shows that 
\begin{equation*}
\int_{0}^\infty \left|n t^{n-1}\left(\int_t^\infty \frac{\nu_n(ds)}{s^n}\right) \right| \,dt < \infty.
\end{equation*}
By the uniqueness of the Stieltjes representation, it follows that
\[
\nu_{n+1}(dt) = nt^{n-1}
\int_t^\infty \frac{\nu_n(ds)}{s^n} \,dt.
\]

By (\ref{Asfa}) we have, for any $x\in X$,
\begin{align*}
f_{n+1}(A)x &= n\int_{0+}^\infty t^{n-1}(1+tA)^{-(n+1)}x\,
\left(\int_t^\infty s^{-n}\nu_n(ds)\right)\, dt  \\
&= n\int_{0+}^\infty  \left(\int_0^s t^{n-1} (1+tA)^{-(n+1)}x\,dt\right)
\, s^{-n}\nu_n(ds) \\
&= n\int_{0+}^\infty  \left(\int_{1/s}^\infty (\tau+A)^{-(n+1)}x\,d\tau\right)
\, s^{-n}\nu_n(ds) \\
&= \int_{0+}^\infty  (1/s+A)^{-n}x\,s^{-n}\nu_n(ds)  \\
&= \int_{0+}^\infty  (1+sA)^{-n}x\,\nu_n(ds)=f_n(A)x.
\qedhere
\end{align*}
\end{proof}

\begin{lemma}\label{FcHc}
Let $f\in \tilde{\mathcal{S}}_{n,b}$, and let $A$ be an injective sectorial  operator on  $X$.  Let $f(A)$ be the bounded operator defined by \eqref{Asfa}.  Then $f$ is regularisable in the holomorphic functional calculus for $A$, and
the operator $f_{hol}(A)$ defined by that calculus satisfies
\begin{equation}\label{phillips0}
f_{hol}(A)= f(A).
\end{equation}
\end{lemma}

\begin{proof}
Let $f\in \tilde{\mathcal{S}}_{n,b}$ with  the representation \eqref{GlF} and let $A \in \Sect(\omega)$ where $\omega<\pi$. Then $f$ is a bounded holomorphic function in the sector $S_\omega$, so $f$ is regularisable by
\[
\tau (z)=\frac{z}{(1+z)^{2}}.
\]
For any $x\in X$, by  (\ref{Cauchy}), (\ref{GlF}) and
the boundedness of $\tau(A)$, we have
\begin{eqnarray*}
(f \tau)(A)x &=& \frac{1}{2\pi i}\int_\Gamma f(z)\tau (z) (z-A)^{-1}x\,dz \\
&=& \frac{1}{2\pi i}\int_\Gamma \left\{\int_0^\infty \frac{\nu(ds)}{(1+zs)^n}\right\}\,\tau(z) (z-A)^{-1}x\,dz \\
&=& \int_0^\infty \left\{\frac{1}{2\pi i}\int_\Gamma  \frac{\tau (z)}{(1+zs)^n}\, (z-A)^{-1}x\,dz\right\}\,\nu(ds) \\
&=& \int_0^\infty \tau(A) (1+sA)^{-n}x\,\nu(ds)  \\
&=& \tau(A)\int_0^\infty  (1+sA)^{-n} x\,\nu(ds).
\end{eqnarray*}
It follows from this and the definition of $f_{hol}(A)$ via the holomorphic functional calculus that $\dom(f(A))=X$ and the equation (\ref{phillips0}) holds.
\end{proof}

We wish to show that the map $S : \tilde S_b \to \mathcal{L}(X)$ defined in (\ref{Asfa}) is an algebra homomorphism.  If $A$ is injective, this follows immediately from \eqref{zv} and Lemma \ref{FcHc}.  When $A$ is not injective we use the following approximation result.

\begin{proposition}\label{ttn01}
Let $f\in \tilde{\mathcal{S}}_{n,b}$. Then
\begin{equation*}
\lim_{\delta\to 0+}\|f(A+\delta)-f(A)\|=0.
\end{equation*}
\end{proposition}

\begin{proof}
Let $f$ have the representation (\ref{compbern1C}), and take $\delta \in (0,1)$ and $x\in X$. We have
\begin{equation}\label{EA0010}
f(A)x-f(A+\delta)x=
\int_0^\infty
F_n(s,\delta)(A)x\,\nu(ds),
\end{equation}
where
\begin{eqnarray*}
F_n(s,\delta)(A)&:=& (1+sA)^{-n}-(1+\delta s+sA)^{-n} \\
&=& \delta s \sum_{k=0}^{n-1} (1+sA)^{-(k+1)}(1+s\delta +sA)^{-(n-k)}.
\end{eqnarray*}
Then, using (\ref{stiloper}) and (\ref{sA02}), we have
\begin{equation*}\label{EA101}
\|F_n(s,\delta)(A)\|\le \delta s \sum_{k=0}^{n-1} M(A)^{n+1} \frac{1}{(1 + \delta s)^{n-k}} \le 
 \frac{M(A)^{n+1}n \delta s}{1+\delta s}.
\end{equation*}
Then we obtain from (\ref{EA0010}) that
\begin{equation*}
\|f(A+\delta)-f(A)\|
\le M(A)^{n+1}n \int_0^\infty \frac{\delta s\,|\nu|(ds)}{1+\delta s} \to 0
\end{equation*}
as $\delta\to0+$, by Lebesgue's dominated  convergence theorem since $|\nu|$ is a finite measure.
\end{proof}

\begin{proposition}\label{homAlg}
For any sectorial operator $A$, the map
\[
S: \tilde{\mathcal{S}}_b \ni f \mapsto f(A)\in \mathcal{L}(X),
\]
defined by \eqref{Asfa}, is an algebra homomorphism.
\end{proposition}

\begin{proof}
Let $f\in \tilde{\mathcal{S}}_{n,b}$ and $g\in \tilde{\mathcal{S}}_{m,b}$, so that $h := f g \in \tilde{\mathcal{S}}_{n+m,b}$.  We need to prove that
\begin{equation}\label{HolHir}
h(A)=f(A)g(A).
\end{equation}

Assume first that $A$ is injective.  Then, by Lemma  \ref{FcHc}
we have the equalities
\begin{equation*}
f(A)=f_{hol}(A),\quad g(A)= g_{hol}(A),\quad h(A)=h_{hol}(A).
\end{equation*}
So (\ref{HolHir}) follows from (\ref{zv}).

Now consider the case when $A$ is not injective. Then, for any $\delta>0$, we can apply the previous case to $A+\delta$ and obtain
\[
h(A+\delta)=f(A+\delta)g(A+\delta).
\]
On letting $\delta\to 0+$ the statement (\ref{HolHir}) follows
 by applying Proposition \ref{ttn01} to the functions
$f\in \tilde{\mathcal{S}}_{n,b}$, $g\in \tilde{\mathcal{S}}_{m,b}$
and $h\in \tilde{\mathcal{S}}_{n+m,b}$.
\end{proof}

\subsection{The complete Bernstein algebra}\label{SubAlg}
In order to include Hirsch's definition of functional calculus for complete Bernstein functions, we need an algebra of functions that contains $\mathcal{CBF}$.  We proceed as follows.

For $n\in \N$ we consider the class $\tilde{\mathcal{T}}_n$ of functions of the form
\begin{equation}
f(z)=a+\int_0^\infty \frac{z^n \mu(ds)}{(1+zs)^n},\qquad z\in \C\setminus (-\infty,0],
\label{sfun0n}
\end{equation}
where $a\in \C$ and  $\mu$ is a complex Radon measure on $[0,\infty)$
such that
\begin{equation*}
\int_0^\infty \frac{|\mu|(ds)}{(1+s)^n}<\infty.
\end{equation*}
Putting $\hat f(z) = f(1/z)$ and comparing with (\ref{salpha}), one sees that $f \in \tilde{\mathcal{T}}_n$ if and only if $\hat f \in \tilde{\mathcal{S}}_n$.  Moreover, $z^\beta \in \mathcal{T}_n$ if $0 < \re\beta< n$.

As a corollary of (\ref{Sinclus}) and Lemma \ref{WiddR} we have the following statement.

\begin{proposition}\label{CSt}
The inclusion
\[
\tilde{\mathcal{T}}_n\,\subset\,\tilde{\mathcal{T}}_m,\quad m>n,
\]
holds. If $f_1\in \tilde{\mathcal{T}}_n$, $f_2\in \tilde{\mathcal{T}}_m$ for some $n,m\in \N$, then
\[
f_1 f_2 \in \tilde{\mathcal{T}}_{n+m}.
\]
\end{proposition}

Thus the set 
\[
\tilde{\mathcal{T}}:=\bigcup_{n\in\N}\,\tilde{\mathcal{T}}_n
\]
 is an algebra.  Note that the class $\tilde{\mathcal{T}}_1$ coincides with the class $\mathcal{T}$ introduced in \cite{Mart} and discussed in the introduction of this paper. The class $\mathcal{CBF}$ of complete Bernstein functions corresponds to (\ref{sfun0n}) when $n=1$, $a\ge0$ and $\mu$ is a positive measure.  So we may call $\tilde{\mathcal{T}}$ the {\it complete Bernstein algebra}.

Let $A$ be a sectorial operator on a Banach
space $X$. By Proposition \ref{homAlg} there is a functional calculus $(\mathcal{E},\mathcal{F},\Phi)$
with the algebra $\mathcal{F}=\mathcal{O}(S_\pi)$ (the space of all
holomorphic functions on the cut plane $z\in \C\setminus
(-\infty,0]$), the subalgebra $ \mathcal{E}=\tilde{\mathcal{S}}_b$,
and the algebra  homomorphism $\Phi=S$, defined by (\ref{Asfa}).
Because the function $\mathbf i(z)\equiv 1$ belongs to $\tilde{\mathcal{S}}_b$ and $\Phi(\mathbf i)=1$, this functional calculus is proper and we can consider the extended calculus.  The following lemma shows that this is a functional calculus {\em for} $A$, and we shall call it the (extended) {\em Stieltjes calculus} for $A$.

\begin{lemma}\label{Aregul0}
Let $A$ be a sectorial operator on the Banach
space $X$. Then any  $f\in \tilde{\mathcal{T}}$ is regularisable in the
extended Stieltjes calculus. More precisely, if $f\in
\tilde{\mathcal{T}}_n$, $n\in \N$, then the function
\[
\psi_n(z):=\frac{1}{(1+z)^n}\in \tilde{\mathcal{S}}_{n,b}
\]
is a regulariser for $f$, such that
\[
\psi_n f\in\tilde{\mathcal{S}}_{2n,b}.
\]
In particular, $\iota$ is regularisable and $\Phi(\iota) = A$.
\end{lemma}

\begin{proof}
First, $\psi_n(z) \in {\mathcal{S}}_{n,b}$,
with the representing measure $\nu$ equal to the unit mass at $1$.  Hence
\[
\Phi(\psi_n)=(1+A)^{-n},
\]
which is injective. Moreover,
\begin{eqnarray}\label{eqbn}
z^n \psi_n(z)&=&\left(1-\frac{1}{1+z}\right)^n\in \tilde{\mathcal{S}}_{n,b},  \\
\Phi(z^n \psi_n(z)) &=& (1 - (1+A)^{-1})^n = (A(1+A)^{-1})^n. \nonumber
\end{eqnarray}
So $z^n$ is regularisable, and $\Phi(z^n) = A^n$.

Let $f\in \tilde{\mathcal{T}}_n$ with the representation (\ref{sfun0n}).  Then
$
f(z)= a+ z^nf_0(z)+f_1(z),
$
where
\[
f_0(z):=\int_0^{1}\frac{\mu(ds)}{(1+zs)^n} \in \tilde{\mathcal{S}}_{n,b},\quad \mbox{and}\quad
f_1(z):=\int_{1+}^\infty\frac{z^n\,\mu(ds)}{(1+zs)^n}.
\]
By (\ref{eqbn}) the function
$\psi_n(z)z^n f_0(z)\in \tilde{\mathcal{S}}_{2n,b}$.
So, it suffices to prove that $\psi_n f_1\in \tilde{\mathcal{S}}_{2n,b}$.
We have
\[
f_1(z)=\int_1^\infty \left(1-\frac{1}{1+zs}\right)^n\,\frac{\mu(ds)}{s^n}
=\sum_{k=0}^n (-1)^k \begin{pmatrix} n\\k \end{pmatrix} f_{1,k}(z),
\]
where
\[
f_{1,k}(z):=\int_1^\infty \frac{\mu(ds)}{(1+zs)^k s^n}
\in \tilde{\mathcal{S}}_{k,b}\subset \tilde{\mathcal{S}}_{n,b},\quad 0\le k\le n.
\]
So
$\psi_n f_1\in \tilde{\mathcal{S}}_{2n,b}$.
\end{proof}

Now let $A$ be a sectorial operator on $X$ and $f\in \tilde{\mathcal{T}}_n$.  By Lemma \ref{Aregul0} we can define
the closed operator $f(A) = f_S(A)$ by the extended Stieltjes calculus as:
\begin{equation} \label{fAesc}
f(A)=[\psi_n(A)]^{-1}[f \psi_n](A) = (I+A)^n [f \psi_n](A),
\end{equation}
Here $(f\psi_n)(A)$ is defined  by (\ref{Asfa}) for $f \psi_n\in\tilde{\mathcal{S}}_{2n,b}$. 

\begin{lemma} \label{core}
If $f \in \tilde{\mathcal{T}}_n$, then  $\dom (A^n)$ is a core for $f(A)$.
\end{lemma}

\begin{proof}  Let
\[
e_k(z):=\psi_n(k^{-1}z)\in \tilde{\mathcal{S}}_{n,b},\quad k\in \N.
\]
By a slight variant of Lemma \ref{Aregul0}, $e_k$ is a regulariser for $f$, so
\[
\ran(e_k(A))=\dom (A^n)\subset \dom(f(A)),
\]
by Proposition \ref{properF}(\ref{errr}).  Moreover, $e_k(A) \to 1$ in the strong operator topology (see \eqref{sot}), so the claim follows from  Proposition \ref{GenerC}(a).
\end{proof}

For $f\in \tilde{\mathcal{T}}_n$ one can define $f(A)$ in an alternative way to (\ref{fAesc}), generalising the definition (\ref{stilclos}) from \cite{Mart} for $n=1$, and hence generalising Hirsch's definition \cite{HirFA} when $f$ is a complete Bernstein function.

Let $f \in \tilde{\mathcal{T}}_n$ have the representation (\ref{sfun0n}).  We define the  operator $\hat f(A)$ by:
\begin{equation}
 \dom(A^n) \ni x \,\mapsto\,
\hat{f}(A)x=ax+\int_0^\infty A^n(1+sA)^{-n}x\,\mu(ds).
\label{stilclos0}
\end{equation}

\begin{proposition}\label{closable}
The operator $\hat{f}(A)$ defined by \eqref{stilclos0} is
closable and
\begin{equation} \label{ClosABC0}
\overline{\hat{f}(A)}=(1+A)^{n}\hat{f}(A)(1+A)^{-n}.
\end{equation}
\end{proposition}

\begin{proof}  Using (\ref{stiloper}) and (\ref{sA0}), we have,  for $x\in \dom(A^n)$,
\begin{multline} \label{ineq}
\|\hat{f}(A)x\|\le |a|\|x\|
+M(A)^n\|A^nx\|\int_0^1\,|\mu|(ds)  \\
 + (1+M(A))^n\|x\|\int_{1+}^\infty \,\frac{|\mu|(ds)}{s^n}. \phantom{XXXXX}
\end{multline}
Since
\[
(1+A)^{-n}\hat{f}(A)x=\hat{f}(A)(1+A)^{-n}x,
\]
 (\ref{ineq}) implies that
\begin{equation*}
\|(1+A)^{-n}\hat{f}(A)x\|\le C\|x\|,\qquad x\in \dom(A^n) = \dom(\hat f(A)).
\end{equation*}
Since $(1+A)^{-n}$ is a bounded injective operator, it follows that 
$\hat f(A)$ is closable and \eqref{ClosABC0} holds.
\end{proof}

By Proposition \ref{closable} we can define a closed operator
$f_{H}(A)$  as
\begin{eqnarray}
\label{HAaa}
f_{H}(A) &:=& \overline{\hat{f}(A)} \;\; = \;\; (1+A)^{n}\hat{f}(A)(1+A)^{-n}, \\
\dom(f_H(A)) &=& \{x\in X:\,\hat{f}(A)(1+A)^{-n}x\in \dom (A^n) \nonumber\}.
\end{eqnarray}
We will show in Theorem \ref{ConHPH} that this definition agrees with the definition by the extended Stieltjes calculus in \eqref{fAesc}.  In particular this will show that the definition \eqref{HAaa} is independent of $n$. As in Proposition \ref{homAlg} it is convenient to use an approximation argument to reduce the general case to the case when $A$ is invertible.

\begin{proposition}\label{ttn}
Let $f\in \tilde{\mathcal{T}}_n$. Then, for each $x\in \dom(A^n)$,
\begin{equation}
\lim_{\delta\to 0+}\|f_H(A+\delta)x-f_H(A)x\|=0. 
\label{conver}
\end{equation}
\end{proposition}

\begin{proof}
Let $f$ have the representation (\ref{sfun0n}). Then  for $x\in \dom(A^n)$ and $\delta \in (0,1)$ we have
\begin{equation}\label{EA00}
f_H(A+\delta)x-f_H(A)x=
\int_0^\infty G_n(s,\delta)(A)x\,\mu(ds),
\end{equation}
where
\begin{eqnarray}
G_n(s,\delta)(A) &:=& (A+\delta)^n(1+\delta s+sA)^{-n}-A^n(1+sA)^{-n} \nonumber \\
&=& \delta
(1+sA)^{-1}(1+\delta s+sA)^{-1} \label{EA05}\\
\nonumber 
&& \phantom{X} \cdot 
\sum_{k=0}^{n-1} [A(1+sA)^{-1}]^k [(A+\delta)(1+s\delta +sA)^{-1}]^{(n-1-k)}.
\end{eqnarray}
Then, using (\ref{stiloper}), (\ref{sA0}), (\ref{sA02}) for $A$ and $A+\delta$, we have
\begin{eqnarray}
\|G_n(s,\delta)(A)x\| &\le& \delta \frac{M(A)^2}{1+\delta s}\sum_{k=0}^{n-1} \frac{(1+M(A))^k}{s^k}
\frac{(1+M(A))^{n-1-k}}{s^{n-1-k}} \|x\| \nonumber \\
\label{EA10}
&=& \frac{C\delta}{(1+\delta s)s^{n-1}}\|x\|,
\end{eqnarray}
where $C=M(A)^2 (1+M(A))^{n-1}n$. Moreover, (\ref{EA05}) gives
\begin{eqnarray*}
\|G_n(s,\delta)(A)x\| &\le& \delta  M(A)^{n+1}\sum_{k=0}^{n-1} \|A^k(A+\delta)^{n-1-k}x\|  \\
&\le& \delta M(A)^{n+1}\sum_{k=0}^{n-1}\sum_{m=0}^{n-1-k}  \begin{pmatrix} {n-1-k} \\m \end{pmatrix} \|A^{m+k}x\| \\
&\le& \delta M(A)^{n+1}\|x\|_{\dom(A^{n-1})}\sum_{k=0}^{n-1}\sum_{m=0}^{n-1-k} \begin{pmatrix} n-1-k\\m \end{pmatrix} \\
&\le& C_0 \delta \|x\|_{\dom(A^{n-1})},
\end{eqnarray*}
with $C_0=2^n M(A)^{n+1}$ and $\|x\|_{\dom(A^{n-1})} = \max_{r=0,1,\dots,n-1} \|A^rx\|$.  From this, (\ref{EA10}) and (\ref{EA00}) we obtain the estimate
\begin{eqnarray*} 
\lefteqn{\|f_H(A+\delta)x-f_H(A)x\|} \\
&\le&
 C_0\delta\|x\|_{\dom(A^{n-1})} \int_0^1 |\mu|(ds)
+C\|x\|\int_{1+}^\infty \frac{\delta\,|\mu|(ds)}{(1+\delta s)s^{n-1}},
\end{eqnarray*}
and (\ref{conver}) follows.
\end{proof}

Now assume that $A$ is injective and sectorial.  There is a third way to define $f(A)$ for $f\in \tilde{\mathcal{T}}_n$ via the holomorphic functional calculus.  For any $\omega\in (0,\pi)$ we have
\begin{equation}
|f(z)|\le  C_\omega \max(1,|z|^n), \quad z\in S_\omega.
\label{maxmin}
\end{equation}
Let $\tau(z) = z(1+z)^{-2}$, so $\tau(A) = A(1+A)^{-2}$ which is injective and $\dom \left(\tau^m(A)^{-1}\right)=\dom (A^m)\cap \ran(A^m)$ for any $m\in\N$ \cite[Proposition 9.4]{Weis}.  Then $\tau^{n+1}$ is a regulariser for $f$ in the holomorphic functional calculus, and so we can define $f_{hol}(A)$ as
\begin{equation}
f_{hol}(A)=[\tau(A)]^{-(n+1)}(\tau^{n+1} f)(A),
\label{fA}
\end{equation}
with
\[
\dom(f_{hol}(A))= \left\{x\in X:(\tau^{n+1} f)(A)x\in \dom(A^{n+1}) \cap \ran(A^{n+1}) \right\}.
\]

\begin{theorem}\label{ConHPH}
Let  $f\in \tilde{\mathcal{T}}_n$  and let  $A$
be an injective, sectorial  operator on  $X$. Let
$f_S(A)$ be defined by the extended Stieltjes calculus as in \eqref{fAesc}, $f_{H}(A)$ be defined by  \eqref{HAaa}, and  $f_{hol}(A)$ be defined by the extended holomorphic calculus as in \eqref{fA}.  Then
\begin{equation}\label{HolST}
f_S(A)=f_{hol}(A)=f_H(A).
\end{equation}
\end{theorem}

\begin{proof}
First note that  
\[
\frac {z}{1+z} \in \tilde{\mathcal{S}}_{1,b}, \qquad  \frac {1}{1+z} \in \tilde{\mathcal{S}}_{1,b}.
\]
Hence
\[
\tau^{n+1}(z)= \left(\frac{z}{1+z}\right) ^{n+1} \left(\frac{1}{1+z}\right)^{n+1} 
\in \tilde{\mathcal{S}}_{2(n+1),b}.
\]
Moreover, 
\[
f(z)\tau^{n+1}(z) = f(z) \psi_n(z) \left(\frac{z}{1+z}\right) ^{n+1} \left(\frac{1}{1+z}\right) \in \tilde{\mathcal{S}}_{b}
\]
by Lemma \ref{Aregul0}.  Thus $\tau^{n+1}$ is a regulariser for $f\in \tilde{\mathcal{T}}_n$ in the Stieltjes calculus as well as the holomorphic calculus.
By  Lemma \ref{FcHc},
\[
[f \tau^{n+1}]_{hol}(A)=[f \tau^{n+1}]_S(A).
\]
So,
\begin{eqnarray} \label{HolSTo}
f_{hol}(A) &=& (\tau^{n+1}(A))^{-1}[f \tau^{n+1}]_{hol}(A), \\
&=& (\tau^{n+1}(A))^{-1}[f \tau^{n+1}]_S(A), \nonumber \\
&=& f_S(A). \nonumber
\end{eqnarray}

Next, let $f$ have the representation \eqref{sfun0n}.  By  (\ref{Cauchy}), we have for any $x \in X$,
\begin{eqnarray*}
[f \tau^{n+1}]_{hol}(A)x &=& \frac{1}{2\pi i}\int_\Gamma f(z)\tau^{n+1} (z) (z-A)^{-1}x\,dz \\
&=&\frac{1}{2\pi i}\int_\Gamma \left\{a+\int_0^\infty \frac{z^n\,\mu(ds)}{(1+zs)^n}\right\}\,\tau^{n+1}(z) (z-A)^{-1}x\,dz \\
&=& \frac{a}{2\pi i}\int_\Gamma \tau^{n+1}(z) (z-A)^{-1}x\,dz \\
&& \phantom{XX}
+\int_0^\infty \left\{\frac{1}{2\pi i}\int_\Gamma  \frac{z^n\tau^{n+1} (z)}{(1+zs)^n}\, (z-A)^{-1}x\,dz\right\}\,\mu(ds)  \\
&=&
a \tau^{n+1}(A)x+\int_0^\infty  A^n(1+sA)^{-n}\tau^{n+1} (A)x\,\mu(ds).
\end{eqnarray*}
From this, the boundedness of the operator $A(1+A)^{-2}$ and
the estimate (\ref{ineq}), we have
\begin{eqnarray*}
[f \tau^{n+1}]_{hol}(A)x &=& \tau^{n+1}(A)\left\{ax+\int_0^\infty  A^n(1+sA)^{-n} x\,\mu(ds)\right\}  \\
&=& \tau^{n+1}(A) f_H(A),\quad x\in \dom(A^n),
\end{eqnarray*}
and then
\begin{equation}\label{coreB}
f_{hol}(A)x=(\tau(A))^{-(n+1)}[f \tau^{n+1}]_{hol}(A)x= f_H(A)x,\quad x\in \dom(A^n).
\end{equation}
By the definition of $f_H(A)$,  $\dom(A^n)$ is a core for $f_H(A)$.  By \eqref{HolSTo} and Lemma \ref{core}, $\dom(A^n)$ is a core for $f_{hol}(A)=f_S(A)$.
Then we obtain the second equality in (\ref{HolST})   from (\ref{coreB}).
\end{proof}

\begin{corollary}\label{coins}
Let  $f\in \tilde{\mathcal{T}}_n$ and let $A$ be a sectorial  operator on  $X$. Then $f_S(A)$ as defined in the extended
Stieltjes calculus coincides with  $f_H(A)$ defined by \eqref{HAaa}.
\end{corollary}

\begin{proof}  Take $x \in \dom(A^n)$, and let $\psi_n(z) = (1+z)^{-n}$.  By Theorem \ref{ConHPH}, for any $\delta>0$,
\begin{equation}\label{InjNo}
f_H(A+\delta)x = f_S(A+\delta)x = [f\psi_n]_S(A+\delta)(1+\delta+A)^{n}x. 
\end{equation}
By Proposition \ref{ttn},
\begin{equation*} 
\lim_{\delta\to0}\,f_{H}(A+\delta)x=f_H(A)x.
\end{equation*}
By Proposition \ref{ttn01},
\[
\lim_{\delta\to 0+}\,\|[f \psi_n]_S(A+\delta)-[f \psi_n]_S(A)\|=0.
\]
Since
\[
\lim_{\delta\to 0+}\,(1+\delta+A)^nx=(1+A)^nx,
\]
on passing to the limit as $\delta\to 0+$ in (\ref{InjNo}) we obtain
\[
f_H(A)x =[f \psi_n]_S(A)(1+A)^n x
=f_S(A)(1+A)^{-n}(1+A)^n x=f_S(A)x.
\]
As in the proof of Theorem \ref{ConHPH},  $\dom(A^n)$ is a core for both $f_S(A)$ and $f_H(A)$, and the claim follows.
\end{proof}

In the light of Theorem \ref{ConHPH} and Corollary \ref{coins}, we may in future write $f(A)$ instead of $f_S(A)$, $f_{hol}(A)$ or $f_H(A)$ for appropriate functions $f$.  When $\alpha>0$ and $f(z) = z^\alpha$, the coincidence of these definitions is the well known fact that classical definitions of the fractional powers $A^\alpha$ agree (see \cite[Proposition 3.1.12]{Ha06}, \cite[Section5.2]{Mart1}).

\subsection{Product formula in the extended Stieltjes calculus}
In this subsection, $A$ will be a sectorial operator
on a Banach space $X$.  For $f\in \tilde{\mathcal{T}}$, $f(A)$ will denote the operator defined by the extended Stieltjes calculus, or alternatively by \eqref{HAaa} (see Corollary \ref{coins}).

If $f \in \tilde{\mathcal{T}}_n$, then $e_k(z) := (1+k^{-1}z)^{-n}$ is a regulariser for $f$.  Moreover $\lim_{k\to\infty} (1+ k^{-1}A)^{-n} = 1$ in the strong operator topology (see \eqref{sot}).  It follows from Proposition \ref{GenerC}(b) that if $f,g\in \tilde{\mathcal{T}}$, then
\[
\overline{f(A)g(A)}=[f g](A).
\]
By (\ref{RKrit}), the product formula (\ref{main}) holds if
\begin{equation*}
\dom([f g](A))\subset \dom (g(A)).
\end{equation*}
We shall prove in Theorem \ref{cor3} that the product formula holds in many cases.  As in Proposition \ref{homAlg} and Corollary \ref{coins} the proof will be by reduction to the case when $A$ is injective, but using a quotient construction.  Passing to $A + \delta$ does not seem helpful in this context.

\begin{lemma}\label{quotient}
Let $A$ be a sectorial operator on $X$, and $f\in \tilde{\mathcal{T}}$.  Consider the quotient space
\[
X_0:=X/\ker{A}
\]
with the canonical quotient map $u: X\mapsto X_0$, and the quotient
operators $A_0$ and  $f_0(A)$ on $X_0$ given by
\begin{eqnarray}
\dom(A_0) &=& u(\dom(A)), \nonumber \\
\label{qsect} A_0 u(x)&=&u(Ax), \qquad x\in \dom(A),  \\
\dom(f_0(A)) &=& u(\dom(f(A))), \nonumber \\
\label{kernelQ}
f_0(A)u(x) &=& u(f(A)x), \qquad x\in \dom(f(A)).
\end{eqnarray}
Then $A_0$ is a sectorial operator on $X_0$ and
\begin{equation}\label{f0}
f(A_0)=f_0(A).
\end{equation}
\end{lemma}

\begin{proof}  Note first that $A_0$ is well-defined with dense domain.  It follows from the definitions of $f(A)$ that $f(A)x\in \ker(A)$ for all  $x\in \ker(A)\subset \dom(A)$, so the quotient operator $f_0(A)$ is correctly defined as a linear operator. Moreover, for $s>0$ the operator $(1+sA)^{-1}$ induces a bounded operator $R_s$ on $X_0$, and it is easy to see that $R_s = (1+sA_0)^{-1}$.  Hence 
\[
\|(1+sA_0)^{-1}\| \le \|(1+sA)^{-1}\| \le M(A).
\]
Thus $A_0$ is sectorial.

Since $(1+sA_0)^{-1}u = u (1+sA)^{-1}$,
it is easy to see that $g(A_0)u = ug(A)$ for all $g \in \mathcal{S}_b$.  So we can apply Proposition \ref{quotfc} with $\mathcal{E} = \mathcal{S}_b$ and  $\Phi$ and $\Phi_0$ the natural functional calculi associated with $A$ and $A_0$ as in \eqref{Asfa}, and $e(z) = (1+z)^{-n}$ (see Lemma \ref{Aregul0}), noting that $e(A) = (1+A)^{-n}$ is the identity map on $\ker u = \ker A$.   This establishes (\ref{f0}). 
\end{proof}

\begin{remark} \label{closedness}
It follows from Lemma \ref{quotient} that the operator $f_0(A)$ is closed.  This can be seen directly, and (\ref{f0}) follows as a corollary  because $f(A_0)$ and $f_0(A)$ are both closed operators and they agree on a common core $\dom(A_0^n) = u(\dom(A^n))$ (see (\ref{HAaa}) and Corollary \ref{coins}).
\end{remark}

We consider the following condition on $f \in \tilde{\mathcal{T}}$:
\begin{equation}\label{tilde0}
\text{$f$ has no zeros in $\C \setminus (-\infty,0]$ and $\tilde{f}(z) := 1/f(1/z)\in \tilde{\mathcal{T}}$.}
\end{equation}
It suffices that $f \in \tilde{\mathcal{T}}$, $f$ has no zeros in $(0,\infty)$ and $\tilde{f}$ (considered as a function on $(0,\infty)$) belongs to $\tilde{\mathcal{T}}$, because $f$ and $1/f$ then have holomorphic extensions to $\C \setminus (-\infty,0]$ and they are mutually reciprocal.  Note that any non-zero function in $\mathcal{CBF}$ satisfies \eqref{tilde0} \cite[p.66]{SSV}.  If $\re\alpha>0$ and $f(z)=z^\alpha$, then $\tilde{f} = f$ and \eqref{tilde0} holds.  For other examples where $f \in \tilde{\mathcal{T}}_1$ and $f$ has no zeros in $(0,\infty)$, see \cite[p.310]{Mart}.

The following result generalises Theorem \ref{prodM} (the case when $f$, $g$ and $f g$ all belong to $\tilde{\mathcal{T}}_1$) and hence the earlier result of Hirsch (when $f$, $g$ and $f g$ belong to $\mathcal{CBF}$).

\begin{theorem}\label{cor3} 
Let $f$, $g\in \tilde{\mathcal{T}}$ and let $A$ be a sectorial
operator. Assume that $f$ satisfies \eqref{tilde0}.  
Then the product formula \eqref{main} holds.
\end{theorem}

\begin{proof}
Let  $f\in \tilde{\mathcal{T}_k}$ and $g\in \tilde{\mathcal{T}}_m$ for some
$k,m\in \N$, so  that  $\psi_k$ is a regulariser for $f$
and $\psi_m$ is a regulariser for $g$ in the Stieltjes calculus, where
\[
\psi_k(z)=\frac{1}{(1+z)^k}\in \tilde{\mathcal{S}}_{k,b}.
\]
By  (\ref{tilde0}), we have that $1/f \in \tilde{\mathcal{S}}_r$ for some $r\in \N$. 
Let
\[
e_3(z)=\frac{z^r}{(1+z)^r}\in \tilde{\mathcal{S}}_{r,b}.
\]
Then $e_3/f \in \tilde{\mathcal{S}}_{2r,b}$,  by Lemma \ref{Aregul0}.

Assume temporarily that $A$ is injective.  Then $e_3$
is a regulariser for $1/f$ in the Stieltjes calculus.  As in the proof of Corollary \ref{cor1}, (\ref{nnA}) holds for $p(z)=(1-z)^{k+m}$, and  by  Corollary \ref{MjA00} we conclude that the product formula (\ref{main}) holds.

Now consider the case when $A$ is not injective.  Let
$X_0:=X/\ker{A}$ with the canonical quotient map $u: X\mapsto X_0$
and the quotient operators $A_0$, $f_0(A)$ on $X_0$ as in
Lemma \ref{quotient}. Similarly, define also the quotient
operators
\begin{eqnarray*}
g_0(A) u(x)&:=&u(g(A))x) \quad \mbox{for}\quad x\in \dom(g(A)),  \\
{}[f g]_0(A)x &:=& u([f g](A)x)\quad \mbox{for}\quad x\in \dom([f g](A)).
\end{eqnarray*}
We can also define the operators $f(A_0)$, $g(A_0)$ and $[f
g](A_0)$ on $X_0$. By  Lemma \ref{quotient},
\begin{equation}\label{eqquot}
f(A_0)=f_0(A),\quad
g(A_0)=g_0(A),\quad [f g](A_0)=[f g]_0(A).
\end{equation}
Applying the previous case to the injective sectorial operator $A_0$, we obtain that
\[
f(A_0)g(A_0)=[f g](A_0),
\]
and then
\begin{equation}\label{q0}
\dom([f g](A_0))\subset \dom (g(A_0)).
\end{equation}
From (\ref{eqquot})
and (\ref{q0})  we have that
\begin{multline*}
u(\dom ([f g](A)))=\dom([f g]_0(A))
=\dom([f g](A_0)) \\
\subset \dom(g(A_0))=\dom (g_0(A))=u(\dom (g(A)).
\end{multline*}
Since $\ker A \subset \dom(g(A))$, it follows that
\[
\dom([f g](A))\subset \dom (g(A)),
\]
so the product formula (\ref{main}) holds.
\end{proof}

\begin{corollary}\label{ProdHH}
Let $A$ be a sectorial operator.  Let $g \in \tilde{\mathcal{T}}$, and for  $j=1,2,\dots,n$, let $f_j\in \tilde{\mathcal{T}}$ satisfy \eqref{tilde0}. Then the product formula
\[
f_1(A)f_2(A)\cdots f_n(A)g(A)=[f_1\cdot f_2\cdots f_n\cdot g](A)
\]
holds.
\end{corollary}

The next result was originally due to Hirsch.

\begin{theorem}\label{cor1H} \cite[Corollary 1]{HirFA}.
If  $f,g, f g\in \mathcal{CBF}$, then
\[
f(A)+[f g](A)=
[f(1+g)](A).
\]
\end{theorem}

The following generalisation of Theorem \ref{cor1H} is a corollary of Theorem \ref{cor3}.

\begin{corollary}\label{cor4}
Let  $f,g\in \tilde{\mathcal{T}}$ be such that $f$ and $1+g$ both satisfy \eqref{tilde0}.
Then
\[
f(A)+[f g](A)=[f+f g](A).
\]
In particular, $f(A) + [f g](A)$ is closed.
\end{corollary}

\begin{proof}
Let $f\in \tilde{\mathcal{T}}_n$ and $g\in \tilde{\mathcal{T}}_m$ for some $n,m\in
\N$, so that $f g\in \mathcal{T}_{n+m}$. By
Proposition \ref{properF}(\ref{sum}), we have
\[
f(A)+[f g](A) \subset [f+f g](A),
\]
so it suffices to prove that
\[
\dom([f+f g](A))\subset \dom (f(A)+[f g](A)).
\]

By the assumptions on $f$ and $1+g$, and Theorem \ref{cor3},  we have that
\begin{equation}\label{faABCD}
f(A)g(A)=[f g](A),
\end{equation}
and
\begin{equation}\label{SumABC}
[f+f g](A)=[f \cdot (1+g)](A)=f(A)(1+g(A))=(1+g(A))f(A).
\end{equation}
From  (\ref{SumABC}) and (\ref{faABCD}) it follows that
\begin{align*}
\dom([f+f g](A))&=\{x\in \dom(g(A)): (1+g(A))x\in \dom(f(A))\}  \\
& \phantom{XX} \cap \{x\in \dom(f(A)): f(A)x\in \dom(g(A))\}  \\
&\subset\dom(f(A))\cap \dom(f(A)g(A))  \\
&=  \dom(f(A))\cap \dom([f g](A))  \\
&=\dom (f(A)+[f g](A)).
\qedhere
\end{align*}
\end{proof}

\section{Hille-Phillips calculus}

\subsection{Extended Hille-Phillips calculus} \label{SubEHF}
In this subsection we recall basic properties of the Hille-Phillips calculus (HP-calculus) for negative generators of bounded $C_0$-semgroups.  This calculus is for certain holomorphic functions on the open right half-plane $\C_+ := S_{\pi/2}$, so it is distinctive from the other calculi in this paper.  The holomorphic calculus of Section \ref{secthol} does provide  a calculus for functions on $\C_+$, but only for sectorial operators $A \in \Sect(\pi/2)$, i.e., for negative generators of bounded holomorphic semigroups.

A complex Radon measure $\mu$ on $[0,\infty)$ is called {\em Laplace-transformable} if
\[ \int_0^\infty e^{-st} \abs{\mu}(\ud{s}) <  \infty \quad
\text{for each $t > 0$}.
\]
The {\em Laplace transform} of a Laplace-transformable complex Radon
measure $\mu$ on $[0,\infty)$ is
\[ (\Lap\mu)(z) := \int_0^\infty e^{-sz} \, \mu(\ud{s}),
\qquad \re z > 0.
\]
This is a holomorphic function on $\C_+$.

If $\mu$ is a positive measure, then $\Lap\mu$ is a {\em completely monotone function} on $(0,\infty)$, i.e., $\Lap\mu$ is a $C^\infty$-function $f$ such that 
$f(z) \ge 0$ and $(-1)^n f^{(n)}(z) \ge 0$ for all $n\in\N, z>0$.  Conversely, any completely monotone function is of the form $\Lap\mu$ for a unique positive, Laplace-transformable, measure $\mu$.

Let $\eM([0,\infty))$ be the space of all complex Radon measures of finite total variation on $[0,\infty)$.  If $\mu \in \eM([0,\infty))$, then $\Lap\mu$ has an extension to a  continuous function on $\cls{\C}_+$. The space
\[ \Wip(\C_+) := \{ \Lap\mu : \mu \in \eM([0,\infty))\}
\]
is a Banach algebra with respect to pointwise multiplication and the
norm
\begin{equation*}
\norm{\Lap \mu}_{\Wip} := \norm{\mu}_{\eM} = \abs{\mu}([0,\infty)),
\end{equation*}
and the Laplace transform
\[ \Lap : \eM([0,\infty)) \pfeil \Wip(\C_+)
\]
is an isometric isomorphism.

Let $-A$ be the generator of a bounded $C_0$-semigroup
$(T(t))_{t\ge 0}$ on a Banach space $X$. Then the mapping
\[ \Phi: g = \Lap {\mu} 
\;\mapsto\; g(A) := \int_0^\infty T(t)\, \mu(d{t})
\]
(where the integral converges in the strong operator topology)
is a continuous algebra homomorphism of $\Wip(\C_+)$ into $\Lin(X)$
satisfying
\begin{equation*}
 \norm{g(A)} \le (\sup_{t\ge 0} \norm{T(t)} ) \, \norm{g}_{\Wip},
\qquad g\in \Wip(\C_+).
\end{equation*}
The homomorphism $\Phi$ is a proper functional calculus, called the {\em Hille-Phillips} (HP) functional calculus, for $A$.   For its basic properties one may consult \cite[Chapter XV]{HilPhi}.

The HP-calculus can be extended by the regularisation method described in Subsection \ref{ABFprel}, and we call this the {\em extended Hille--Phillips calculus} for $A$.  If $f: \C_+ \to \C$ is holomorphic and there exists
a function $e\in \Wip(\C_+)$ with $ef \in \Wip(\C_+)$ and the
operator $e(A)$ is injective, then
\begin{eqnarray*}
 f(A) &:=& e(A)^{-1} \, (ef)(A), \\
 \dom (f(A))&:=& \{x \in X : (ef)(A)x \in \ran(e(A)) \}.
\end{eqnarray*}

We will apply this regularisation approach to the study of
operator Bernstein functions.  A function $f \in C^\infty(0,\infty)$ is a {\em Bernstein function}, in short $f \in \mathcal{BF}$, if $f \ge 0$ and $f'$ is a completely monotone function.   We refer to \cite{SSV} for details about Bernstein functions, but we note that any complete Bernstein function is a Bernstein function.

The following facts were proved in \cite[Lemma 2.5]{GHT}.

\begin{lemma}\label{bf-fc} 
Every Bernstein function $f$ can be written in the form
\[
 f(z) = g_1(z) + z\, g_2(z), \qquad z >0,
\]
where $g_1, g_2 \in \Wip(\C_+)$.  Moreover, $(1+z)^{-1}$ is a regulariser for $f$ in the HP-calculus.
\end{lemma}

If $f$ is a Bernstein function, then $f(A)$ as defined by the extended HP-calculus coincides with the definition of functional calculus given in \cite[Section 12.2]{SSV}, and in particular the extended HP-calculus agrees with the Hirsch calculus when $f \in \mathcal{CBF}$ (see \cite[Corollary 2.6, Remark 2.7]{GHT}).

A non-zero Bernstein function $f$ is said to be a {\em special Bernstein function}, written $f \in \mathcal{SBF}$, if $z/f(z)$ is a Bernstein function. 
Since $f \in \mathcal{CBF}$ if and only if $z/f(z) \in \mathcal{CBF}$, any complete Bernstein function is a special Bernstein function.  We refer to \cite[Chapter 10]{SSV} for more details about the class $\mathcal{SBF}$ and its relation to $\mathcal{BF}$ and $\mathcal{CBF}$.

\begin{proposition}\label{PrS}
Let $f$ be a non-zero special Bernstein function. Then
\begin{equation}
\frac{z}{(1+z)f(z)}\in \Wip(\C_+).
\label{wipp}
\end{equation}
\end{proposition}

\begin{proof}
Applying Lemma \ref{bf-fc} to the Bernstein function $z/f(z)$, we have
\[
\frac{z}{(1+z)f(z)}=\frac{1}{1+z} g_1(z) + \frac{z}{1+z}g_2(z),
\]
where $g_1(z)$, $g_2(z)$, $(1+z)^{-1}$, $z(1+z)^{-1}$ all belong to the algebra $\Wip(\C_+)$.
\end{proof}

\subsection{Product formula in the extended HP-calculus}

Proposition \ref{properF}(\ref{grrr}) provides the following product rule for the extended HP-calculus:
if $f$ is regularisable and $g\in \Wip(\C_+)$,
then
\begin{equation}\label{hpfc.e.prod}
 g(A) f(A) \subseteq f(A) g(A) = (fg)(A).
\end{equation}
We shall now extend this product rule.

The following statement is a version of Lemma \ref{quotient} for semigroup generators.

\begin{lemma}\label{QSem}
Let $-A$ be the generator of a bounded $C_0$-semigroup
$(T(t))_{t\ge 0}$ on $X$. Let $X_0:=X/\ker{A}$ with the canonical quotient map $u: X\mapsto X_0$ and let $A_0$ be the quotient operator defined by \eqref{qsect}. Then $-A_0$ is the generator of a bounded $C_0$-semigroup
$(T_0(t))_{t\ge 0}$ on $X_0$ given by
\begin{equation}\label{TTQ}
T_0(t)u(x)=u(T(t)x),\quad x\in X.
\end{equation}
Moreover, if $f$ is regularisable in the HP-calculus with
regulariser 
\[
e(z)=\psi_k(z)=\frac{1}{(1+z)^k}
\]
for some $k\in \N$,  and $f_0(A)$ is the quotient operator on $X_0$
defined by \eqref{kernelQ}, then
\begin{equation}\label{ZZ}
f(A_0)=f_0(A).
\end{equation}
\end{lemma}

\begin{proof}
The proof of  (\ref{TTQ}) can be found in \cite[p. 61]{Engel}.  It is easily verified that
$g(A_0)u = ug(A)$ for $g \in \Wip(\C_+)$, and then (\ref{ZZ}) follows from Proposition \ref{quotfc} applied to the HP-calculi for $A$ and $A_0$.
\end{proof}

The following statement is a variant of \cite[Theorem 12.22(v)]{SSV}.    In particular, it applies when $f \in \mathcal{SBF}$ and $g \in \mathcal{BF}$, without assuming that $fg \in \mathcal{BF}$.

\begin{theorem}\label{generator}
Let $-A$ be the generator of
a bounded $C_0$-semigroup on $X$.
Let $f\in \mathcal{SBF}$ and let  $g$ be regularisable in the HP-calculus with regulariser
\[
\psi_k(z)=\frac{1}{(1+z)^k},
\]
for some $k \in \N$.  Then the product formula \eqref{main} holds.
\end{theorem}

\begin{proof}
First we assume that $A$ is injective. Then we can apply
Theorem \ref{MjA} for the extended HP-calculus. In this
situation, by  Lemma \ref{bf-fc} and Proposition \ref{PrS}
we can choose the functions $e_1,e_2,e_3\in \Wip(\C_{+})$, as
\[
e_1(z)=\frac{1}{1+z},\quad e_2(z)=\frac{1}{(1+z)^k}, \quad e_3(z)=\frac{z}{1+z},
\]
and $p(z) = (1-z)^{k+1}$ so
\[
p(e_3(A))=(1+A)^{-(k+1)}=[e_1 e_2](A).
\]
So Theorem \ref{MjA} shows that the product formula (\ref{main}) holds when $A$ is injective.

When $A$ is not injective, we consider the quotient space
$X_0:=X/\ker{A}$ with the canonical quotient map $u: X\mapsto X_0$
and the quotient operators $A_0$ and $f_0(A)$ as in Lemma \ref{QSem}, so that
$-A_0$ is the generator of a bounded $C_0$-semigroup
$(T_0(t))_{t\ge 0}$ on $X_0$. Similarly, define  also the quotient
operators
\[
g_0(A) u(x):=u(g(A)x)\quad \mbox{for}\quad x\in \dom(g(A)),
\]
\[
[f g]_0(A)x:=u([f g](A)x)\quad \mbox{for}\quad x\in \dom([f g](A)).
\]
We can also define the operators $f(A_0)$, $g(A_0)$ and $[f
g](A_0)$ on $X_0$ by the extended HP-calculus. By  Lemma
\ref{QSem} we have
\begin{equation*}
f_0(A)=f(A_0),\;\;g_0(A)=g(A_0),\;\;[f g]_0(A)=[f g](A_0).
\end{equation*}
Since the operator $A_0$ is injective, the injective case gives the product formula
\[
f(A_0)g(A_0)=[f g](A_0),
\]
so that
\[
\dom([f g](A_0)) \subset \dom (g(A_0)).
\]
As in the proof of Theorem \ref{cor3}, this leads to 
\[
\dom([f g](A))\subset \dom (g(A)).
\]
Thus the product formula (\ref{main}) holds.
\end{proof}

\begin{corollary}\label{fprodg}
Let $f_1,f_2,\dots,f_n\in \mathcal{SBF}$  and  assume that  $g$
is regularisable in the HP-calculus with regulariser $\psi_k$.
 If  $-A$ is the generator of a bounded
$C_0$-semigroup on $X$, then the product formula
\[
f_1(A)f_2(A)\cdots f_n(A)g(A)=
[f_1\cdot f_2\cdot\cdots f_n\cdot g](A)
\]
 holds.
\end{corollary}

\begin{example}\label{ESpec}
The results of this section do not hold for general Bernstein functions $f$. Indeed, let
\begin{equation*}
f_1(z)=1-e^{-z} \in \mathcal{BF}, \qquad g(z) = \frac{z}{(1+z)f_1(z)}.
\end{equation*}
Suppose that (\ref{wipp}) is  true for $f=f_1 \in \mathcal{BF}$.  Then $g \in \Wip(\C_+)$, so it has a continuous extension to $\overline\C_+$ satisfying
\[
(1+z)f_1(z)g(z) = z, \qquad z \in \overline\C_+.
\]
But $f_1(2\pi i )=0$  and we have  a contradiction.

Now let $-A$ be the generator of a bounded $C_0$-semigroup $(T(t))_{t\ge 0}$ and let $f_2$ be the (complete) Bernstein function $f_2(z)=z$. Then
\[
f_1(A)f_2(A)=(1-T(1))A,\quad \dom(f_1(A)f_2(A))=\dom(A).
\]
Since $f_1(A)$ is a bounded operator,
\[
f_2(A)f_1(A)=[f_1 f_2](A)=A(1-T(1)),
\]
with
\[
\quad \dom([f_1 f_2](A))=\{x\in X:\,(1-T(1))x\in \dom(A)\}.
\]
So, in this case the product formula (\ref{main}) holds if and
only if
\begin{equation}\label{IfIf}
x-T(1)x\in \dom(A)\Rightarrow
x\in \dom(A).
\end{equation}
There exist  bounded $C_0$-semigroups
$(T(t))_{t\ge 0}$, such that (\ref{IfIf}) is not true. For
example, let  $ X:=L^2[0,1]$ and
\[
(Ay)(s)=y'(s),\quad s\in [0,1],
\]
with
\[
\dom(A)=\{y\in X:\; \text{$y$ absolutely continuous},\;
y'\in X,\;y(0)=y(1)\}.
\]
Then $A$ is an unbounded operator  and $-A$ generates the periodic
shift $C_0$-semigroup, and $T(1)=1$. So, in this case we
have
\[
f_1(A)f_2(A)x=0,\quad x\in \dom(f_1(A)f_2(A))=\dom(A).
\]
On the other hand,
\begin{multline*}
f_2(A)f_1(A)x=[f_1 f_2](A)x=0, \\
x\in \dom([f_1 f_2](A))=\dom(f_2(A)f_1(A))=X,
\end{multline*}
so
$\overline{f_1(A)f_2(A)}=[f_1 f_2](A),$ but $f_1(A)f_2(A)\not=[f_1 f_2](A)$.
\end{example}

Complete Bernstein functions satisfy the following property \cite[Proposition 7.10]{SSV}:
\begin{equation}  \label{gm}
f_1, f_2\in \mathcal{CBF}\,\Rightarrow\, \text{$f_1^{\alpha} f_2^{1-\alpha}\in \mathcal{CBF}$ for all $\alpha\in (0,1)$}.
\end{equation}
This leads to another possible definition 
 of $(f_1  f_2)(A)$ for $f_1,f_2\in \mathcal{CBF}$ (see \cite[Definition 12.26]{SSV}):
\begin{equation}\label{ShP}
(f_1 f_2)(A):=[(f_1^{1/2} f_2^{1/2})(A)]^2.
\end{equation}
It is proved  in \cite[Corollary
12.27]{SSV} that if $-A$ is the generator
of a uniformly bounded $C_0$-semigroup, and $f_1,f_2\in
\mathcal{CBF}$, then this definition satisfies
\[
(f_1 f_2)(A)=f_1(A)f_2(A)=f_2(A)f_1(A).
\]
It follows from Theorem \ref{generator} that the definition 
(\ref{ShP}) coincides with the definition of the 
operator $(f_1 f_2)(A)$  by the extended HP-calculus.

\begin{example}
No statement similar to \eqref{gm} is true for Bernstein functions.   
Indeed, let $f_1(z)=1-e^{-z}$  and  $f_2(z)=z$, as in Example \ref{ESpec}, so $f_1 \in \mathcal{BF}$ and $f_2 \in \mathcal{CBF}$.  Suppose (for a contradiction) that
\[
F(z):=f_1^\alpha(z) f_2^\beta(z)\in \mathcal{BF}
\]
for some $\alpha>0$, $\beta>0$.
Then the derivative $F'$ is a completely monotone function, so
\begin{equation}  \label{CompM}
F'(z)=z^\beta (1-e^{-z})^\alpha\left[\frac{\alpha}{e^z-1}+\frac{\beta}{z}\right]
=\int_0^\infty e^{-zt} \nu(dt),  \quad z \in \C_+,
\end{equation}
for some positive, Laplace-transformable, measure $\nu$ on $[0,\infty)$.
This gives
\begin{equation} \label{CompM2}
|F'(\tau+is)|\le F'(\tau),\quad \tau>0,\quad s\in\R.
\end{equation}
Putting $s=2\pi$ in \eqref{CompM2} and using \eqref{CompM}, we obtain
\[
|\tau+2\pi i |^\beta \left[\frac{\alpha}{e^{\tau}-1}-\frac{\beta}{|\tau+2\pi i|}\right]
\le \tau^\beta \left[\frac{\alpha}{e^{\tau}-1}+\frac{\beta}{\tau}\right],\quad \tau>0,
\]
and letting $\tau\to 0+$ we obtain a  contradiction.
\end{example}


\begin{thebibliography}{99}


\bibitem{BCT} C.J.K. Batty, R. Chill and Yu.\ Tomilov, \emph{Fine scales of decay of operator semigroups}, J. Eur.\ Math.\ Soc., to appear.

\bibitem{Clark}
S. Clark, \emph{Sums of operator logarithms}, Q. J. Math.\ 60 (2009), 413--427.

\bibitem{deLau95}
R. deLaubenfels, \emph{Automatic extensions of functional
calculi}, Studia Math.\ {114} (1995), 237--259.


\bibitem{Engel}
K.-J. Engel and R.  Nagel,
\emph{One-parameter semigroups for linear evolution equations}.
Graduate Texts in Mathematics \textbf{194}, Springer-Verlag, New York, 2000.

\bibitem{GHT} A. Gomilko, M. Haase, and Yu.\ Tomilov, \emph{Bernstein functions and  rates in mean ergodic theorems
for operator semigroups,} J. d'Analyse Mathematique {118} (2012), 545--576.

\bibitem{Ha05}
M. Haase, \emph{A general framework for holomorphic functional calculi},
Proc. Edinb. Math. Soc. (2) 48 (2005), 423--444.

\bibitem{Ha06} M. Haase, \emph{The Functional Calculus for Sectorial Operators}.
Operator Theory: Advances and Applications \textbf{169},
Birkh\"auser,  Basel, 2006.

\bibitem{HilPhi}
E. Hille and R.S. Phillips, \emph{Functional Analysis and
Semi-Groups,} 3rd printing of rev.\ ed.\ of 1957, Colloq. Publ.
\textbf{31}, AMS, Providence, RI, 1974.

\bibitem{HirschInt}
F. Hirsch, \emph{Int\'egrales de r\'esolvantes et calcul symbolique}, Ann. Inst. Fourier (Grenoble) 22 (1972), 239-264.

\bibitem{HirFA}
F. Hirsch, \emph{Domaines d'op\'erateurs repr\'esent\'es comme int\`egrales de r\'esolvantes},
J. Functional Analysis 23 (1976), 199--217.

\bibitem{widder}
I.I. Hirschman and D.V. Widder,  \emph{The convolution transform}. Princeton University Press, Princeton, N. J., 1955.

\bibitem{Karp}
D. Karp and E. Prilepkina,
\emph{Generalized Stieltjes transforms: basic aspects.}
 arXiv:1111.4271.

\bibitem{Weis}  P.C. Kunstmann and L. Weis, \emph{Maximal $L_p$-regularity for parabolic equations, Fourier multiplier theorems and $H^\infty$-functional calculus}, in: \emph{Functional analytic methods for evolution equations}.
Lecture Notes in Math.\ \textbf{1855}, Springer, Berlin, 2004,  65--311.

\bibitem{LLL}
F. Lancien, G. Lancien and C. Le Merdy, \emph{ A joint functional calculus for sectorial operators with commuting resolvents}. Proc.\ London Math. Soc. 77 (1998), 387--414.

\bibitem{Mart}
C. Martinez and  M. Sanz,
\emph{An extension of the Hirsch symbolic calculus},
 Potential Anal.\ 9 (1998), 301--319.

\bibitem{Mart1}
C. Martinez and M. Sanz,
\emph{The Theory of Fractional Powers of Operators}.
North-Holland Publishing Co., Amsterdam, 2001.

\bibitem{Prudnikov1}
A. P. Prudnikov, Yu.A. Brychkov and O.I. Marichev, 
\emph{Integrals and series. Elementary functions,} (Russian). Nauka, Moscow, 1981.


\bibitem{Sch}
R.L. Schilling, \emph{Subordination in the sense of Bochner and a related functional calculus,} J. Austral. Math. Soc. Ser. A {64} (1998), 368–-396.

\bibitem{SSV}
 R.L. Schilling, R. Song, and Z. Vondra{\v{c}}ek,  \emph{Bernstein functions}.
 de Gruyter Studies in Mathematics \textbf{37}, Walter de Gruyter, Berlin, 2010.

\end{thebibliography}
\end{document}